\documentclass[11pt,reqno]{amsart}
\usepackage{amssymb}
\usepackage[dvips]{epsfig}
\usepackage{graphicx}
\usepackage{color}
\usepackage{amsmath}
\usepackage{amssymb}
\usepackage{amsfonts}
\usepackage{amsthm}
\usepackage{bbm}
\usepackage[thicklines]{cancel}

\newtheorem{theorem}{Theorem}[section]
\newtheorem*{theorem-a}{Theorem A}
\newtheorem*{theorem-b}{Theorem B}
\newtheorem{lemma}[theorem]{Lemma}
\newtheorem{proposition}[theorem]{Proposition}
\newtheorem{corollary}[theorem]{Corollary}
\theoremstyle{remark}
\newtheorem{remark}{Remark}[section]
\theoremstyle{definition}

\numberwithin{equation}{section}

\title[Decay Properties for solutions to the S-KdV  system]{On long time behavior of solutions of the Schr\"odinger-KdV system with and  without resonant  interactions}
\author[Felipe Linares]{Felipe Linares\,$^{*}$} \thanks{$^{*}$Corresponding author. IMPA, Instituto Matematica Pura e Aplicada, Estrada Dona Castorina 110, 22460-320, Rio de Janeiro, RJ, Brazil, E-mail: {linares@impa.br}.}
\author[Deqin Zhou]{Deqin Zhou\,$^{\dag} $} \thanks{$^{\dag}$College of
Mathematics and Statistics, Chongqing University, Chongqing 401331,
China, E-mail: {deqinzhou1952@126.com}}
\keywords{Asymptotic behavior, Schr\"odinger equation, KdV equation}
\subjclass{35Q53, 35B35, 35B40, 35Q55}

\begin{document}

\maketitle


\begin{abstract}
We consider the long time behavior of the solutions of the coupled Schr\"odinger-KdV systems
\begin{eqnarray*}
\left\{
\begin{array}{llll}i\partial_tu+\partial^2_xu=\alpha uv+\beta u|u|^2,\hskip30pt (x,t)\in \mathbb{R}\times \mathbb{R}^{+},\\
\partial_tv+\partial^3_xv+v\partial_xv=\gamma \partial_x(|u|^2), \hskip20pt (x,t)\in \mathbb{R}\times \mathbb{R}^{+},\\
u, v)|_{t=0} =(u_{0}, v_{0}).
\end{array}
\right.
\end{eqnarray*}
We show that global solutions to this system satisfy locally energy decay in a suitable interval, growing
unbounded in time, in two situations. In the first case, we regard the parameter vector $(\alpha,\beta,\gamma)\in \mathbb{R}^{+}\times \overline{\mathbb{R}^{+}}\times \mathbb{R}^{+}$ without any size assumption on the initial data in $ H^{1}(\mathbb{R})\times H^{1}(\mathbb{R})$. In the second one, we consider the parameter vector
$(\alpha,\beta,\gamma)\in \mathbb{R}^{+}\times \mathbb{R}^{-}\times \mathbb{R}^{+}$. In this case, we give a \lq\lq smallness" criterion involving the product of the parameter $-\beta$ and  a constant depending on the initial data in $H^{1}(\mathbb{R})\times H^{1}(\mathbb{R})$. Our results answer positively the open
questions raised in [F. Linares, A. J. Mendez, SIAM J. Math. Anal. 53(2021) 3838-3855]. We use new ideas and
different techniques from the latter paper.
\end{abstract}



\section{Introduction}
\hspace*{\parindent} We consider the Schr\"odinger-Korteweg-de Vries (Schr\"odinger-KdV)  system
\begin{eqnarray}\label{Equ(0.1)}
\left\{
\begin{array}{llll}i\partial_tu+\partial^2_xu=\alpha uv+\beta u|u|^2, \hskip30pt (x,t)\in \mathbb{R}\times \mathbb{R}^{+},\\
\partial_tv+\partial^3_xv+v\partial_xv=\gamma \partial_x(|u|^2), \hskip20pt (x,t)\in \mathbb{R}\times \mathbb{R}^{+},\\
(u, v)|_{t=0} =(u_{0}, v_{0}),\\
\end{array}
\right.
\end{eqnarray}
where the given coefficients $\alpha,\beta,\gamma\in \mathbb{R}$, and the solution $(u(x,t),v(x,t))\in \mathbb{C}\times \mathbb{R}$.
The system \eqref{Equ(0.1)} models the interactions between short and long waves, such as capillary and gravity water waves, ion-acoustic and electron-plasma waves. Especially, the case $\beta=0$ models the resonant interactions and the case $\beta\neq 0$ models the non-resonant interactions. See \cite{JPSJ1975Kawahara,PRL1974Mima,JFA1998Ponce,DIE2010Wuyifei} for more details of the Schr\"odinger-Korteweg-de Vries  system \eqref{Equ(0.1)} in fluid mechanics and plasma physics.

\indent The well-posedness issues for the initial value problem (IVP) associated to the KdV equation, the Schr\"odingder equation, and the Schr\"odinger-KdV system have been extensively studied in the past three decades (see \cite{CPAM1993KPV,JAMS1996KPV,AM2019Visan,M2009Linares,M2006Tao} for the KdV equation,  \cite{AMS1999Bourgain,E1987Tsutsumi,M1999Sulem} for the Schr\"odinger equation, \cite{PAMS1997Ponce,TAMS2007Linares,JDE2010Guozihua,DIE2005Pecher,N2019Linares,MSA1993Tsutsumi,DIE2010Wuyifei} for the Schr\"odinger-KdV system, as well as the latest well-posedness result in \cite{Arxiv2024Linares} for the Schr\"odinger-KdV system).  We are mainly concerned with the long-time behavior of the Schr\"odinger-KdV system \eqref{Equ(0.1)} in this paper.

We notice that while the KdV equation and the cubic Schr\"odinger equation are completely integrable, the
Schr\"odinger-KdV system is not. Even though solutions of the later system enjoy the following invariants:

\begin{equation}\label{cons1}\mathcal{M}(t):=\int_{-\infty}^{+\infty}|u|^{2} d x=\mathcal{M}(0), \end{equation}
\begin{equation}\label{cons2}
\mathcal{Q}(t):=\int_{-\infty}^{+\infty}\left\{\alpha v^{2}+2 \gamma \operatorname{\Im}\left(u \overline{\partial_{x} u}\right)\right\} d x=\mathcal{Q}(0),
\end{equation}
\begin{equation}\label{cons3}
\mathcal{E}(t):=\int_{-\infty}^{+\infty}\left\{\alpha \gamma v|u|^{2}-\frac{\alpha}{6} v^{3}+\frac{\beta \gamma}{2}|u|^{4}+\frac{\alpha}{2}\left|\partial_{x} v\right|^{2}+\gamma\left|\partial_{x} u\right|^{2}\right\} d x=\mathcal{E}(0),
\end{equation}
where ${\Im}u$ stands for the imaginary part of $u$.

These quantities are the key tools to establish the global well-posedness results  above mentioned and they are
also useful in our arguments in this work. In addition, the global results cited above depend on the parameters
$\alpha, \gamma$ to be both positive or both negative. In the discussion below this observation is important.

\vspace{.5cm}

It is well known that both the Gardner equation
\begin{equation*}
\partial_tu+\partial^3_xu+\partial_x(u^2+\mu u^3)=0, \quad \mu\in \mathbb{R},\quad \mu \neq 0,
\end{equation*}
and the modified KdV equation
\begin{equation*}
\partial_tu+\partial^3_xu+\partial_xu^3=0
\end{equation*}
have breathers. However,  Mu\~ noz and Ponce \cite{CMP2019Ponce}  have  proved that there is a large class of the generalized KdV equations
\begin{equation}\label{Equ(0.5)}
\partial_tu+\partial^3_xu+\partial_x(f(u))=0, \quad f(u)=u^k+o(|u|^k),\quad k=2,3,\cdots,
\end{equation}
where $f(u)$ is a polynomial type, which does not have breathers under the smallness assumption
$\underset{t\in \mathbb{R}}{\sup}\|u(t)\|_{H^1(\mathbb{R})}<\epsilon$. Moreover, Mu\~noz and Ponce \cite{CMP2019Ponce} have proved that as time goes to infinity,  the  solutions to \eqref{Equ(0.5)} decay to zero locally in an interval of space depending on time. More precisely,
\begin{equation*}
\lim_{t\rightarrow \pm\infty}\int_{|x|\lesssim \frac{|t|^{\frac{1}{2}}}{\log|t|}}u^2+|\partial_xu|^2dx=0.
\end{equation*}

Later, Mendez-Mu\~ noz-Poblete-Pozo \cite{CPDE2021MMPP} extended the decay interval from $\{x:|x|\lesssim \frac{|t|^{\frac{1}{2}}}{\log |t|}\}$ to $\{x:|x|\lesssim |t|^b \}$  with $0< b<\frac{2}{3}$ in the sense of the limit inferior.  As far as we know, this is the largest local decay interval for the solution of a single real-valued KdV equation. It is still unknown whether there is local energy decay or increase in the interval $\{x:|x|\gtrsim |t|^b\}$.  Recently, Linares-Mendez \cite{SIAM2021Linares} have considered the  Schr\" odinger-Korteweg-de Vries system \eqref{Equ(0.1)}
and  proved rigorously \begin{equation}\label{SK}\lim_{t\rightarrow +\infty} \inf\int_{|x|\lesssim t^{p}}|u|^2+v^2+|\partial_xu|^2+(\partial_xv)^2dx=0,\quad 0<p<\frac{2}{3}, \quad (u_0,v_0)\in (H^1(\mathbb{R}))^2,\end{equation}  with the restriction that the parameters $\alpha$ and $\gamma$ in \eqref{Equ(0.1)}  are  negative.

\medskip

But when the parameters $\alpha$ and $\gamma$ are  positive,  Linares-Mendez's method does not applied. More precisely, there is not a fixed sign in the integrand (see (2.5) in \cite{SIAM2021Linares}) that
\begin{equation}\label{dif}
\int_{t\geq 2}\frac{1}{t\ln t}\int_{\mathbb{R}}(\frac{v^2}{2}-\gamma |u|^2)\frac{1}{\cosh (\frac{x(\ln t)^{q_1} }{t^{p_1}})}\frac{1}{\cosh (\frac{x(\ln t)^{q_1p_2}}{t^{p_1p_2}})}dxdt<\infty,
\end{equation}
with $p_1\in (0,\frac{2}{3})$, $q_1>0$ and $p_2>1$, if the parameter $\gamma$ is positive.  In \cite{SIAM2021Linares},
the authors raise the question: whether we still have local energy decay property when $\gamma$ and $\alpha$ are positive. One aim in this paper is to answer affirmatively this question by introducing new ideas and techniques.

\medskip


Before going on we briefly discuss the main motivation and the meaning of the quantities involved in our results.

Following the arguments introduced by Mu\~noz and Ponce in \cite{CMP2019Ponce} we can show that the Schr\" odinger-Korteweg-de Vries system
does not have periodic breathers.

In fact, consider the case that $\alpha$ and $\gamma$ are negative. Then multiplying the second equation in  \eqref{Equ(0.1)} by $x$ and integrating in space the result, we deduce that
\begin{equation}\label{momen-0}
\frac{d}{dt} \int_{\mathbb R} xv(x,t)\,dx= \int_{\mathbb R} \Big(\frac{v(x,t)^2}{2}-\gamma\,|u(x,t)|^2\Big)\,dx=\frac12\|v(t)\|^2-\gamma\|u_0\|^2,
\end{equation}
where \eqref{cons1} is employed. Integrating in $t$ leads to
\begin{equation}\label{momen-1}
B(t):=\int_{\mathbb R} xv(x,t)\,dx= \int_0^t\Big( \frac12\|v(t)\|^2-\gamma\|u_0\|^2 \Big) dt+c.
\end{equation}
If we suppose that the solution $(u,v)$ of \eqref{Equ(0.1)} is periodic of period $T$, then the function $B(t)$ cannot be periodic since the integrand is not negative in the integral on the right hand side of the equality.

\medskip

Now, we consider the case that $\alpha$ and $\gamma$ are positive.  We multiply the first equation in \eqref{Equ(0.1)} by $2\gamma x\overline{u}$ integrate in space and take the real part to deduce
\begin{equation}\label{momen-2}
\frac{d}{dt} \int_{\mathbb R} x|u(x,t)|^2\,dx= -2\int_{\mathbb R} \Im(u\,\overline{\partial_xu})(x,t)\,dx=-\frac{\alpha}{\gamma}\|v(t)\|^2+\frac{1}{\gamma}\mathcal{Q}(0),
\end{equation}
where $\mathcal{Q}(0)$ is as in \eqref{cons2}.

Next, we multiply the identity \eqref{momen-0} by $-\frac{2\alpha}{\gamma}$ and minus the identity \eqref{momen-2} to yield
\begin{equation}\label{momen-3}
\frac{d}{dt}\int_{\mathbb R} -\frac{2\alpha}{\gamma} x v(x,t)+ x|u(x,t)|^2\,dx= -\frac{1}{\gamma}\mathcal{Q}(0)-2\alpha \|u_0\|^2.
\end{equation}

Integrating in time the identity \eqref{momen-3} we find that
\begin{equation}\label{momen-4}
\begin{split}
F(t):=\int_{\mathbb R}  -\frac{2\alpha}{\gamma} xv(x,t)+x|u|^2(x,t)\,dx= -\Big(\frac{1}{\gamma}\mathcal{Q}(0)+2\alpha \|u_0\|^2\Big)t+c.
\end{split}
\end{equation}
Assuming that the solution $(u,v)$ of \eqref{Equ(0.1)} is periodic function with period $T$ and $\frac{1}{\gamma}\mathcal{Q}(0)+2\alpha \|u_0\|^2\neq 0$, then we  conclude that $F(t)$  cannot be periodic in time otherwise we are lead to a contradiction. This shows that the system \eqref{Equ(0.1)} has no breather solutions in both situations. \footnote{For the existence and simulations of bright and dark solitons, and other special solutions to the Schr\"odinger-KdV system \eqref{Equ(0.1)}, one may refer to \cite{OQE2021ASAO} and \cite{MS2024Kumar}.}

We can make sense of the quantities involved in the arguments above by applying the theory established in \cite{N2019Linares}. There, it was proved that
solutions of the Schr\"odinger-KdV system satisfy,
$(u,v)\in C([0,T]: H^{s+1/2, r_1}(\mathbb R) \times H^{s ,r_2}(\mathbb R))$, $s+1/2>r_1$, $s>2r_2$ with $s>3/4$,
and $r_1, r_2$ be positive real numbers.

 \vspace{3mm}

 Now, we present our main results in this work. We start with the following lower-order ``mixed" local energy decay criteria for the Schr\"odinger-KdV system \eqref{Equ(0.1)} with or without resonant interactions.

\begin{theorem}\label{th1} Let $\alpha$, $\gamma\in \mathbb{R}^{+}$ and $\beta\in \mathbb{R}$. Assume that $(u,v)$ is the solution to the Schr\"odinger-KdV system \eqref{Equ(0.1)} with $(u_0,v_0)\in (H^1(\mathbb{R}))^2$ satisfying
\begin{equation}\label{sm1}
(u,v)\in (C(\mathbb{R}^{+};H^1(\mathbb{R}))^2.
\end{equation}
Then
\begin{equation}\label{decay1}
\begin{split}
\lim _{t\rightarrow \infty} \inf \int_{| x | \lesssim t^{p}} |\frac{v^2}{2}-\gamma|u|^2| d x=\lim _{t\rightarrow \infty} \inf \int_{| x | \lesssim t^{p}} |u(\alpha v+\beta|u|^2)| d x=0, \hskip10pt \forall p\in (0,\frac{2}{3}).
\end{split}
\end{equation}
\end{theorem}
\vspace{0.5cm}
As discussed before, we cannot directly determine the sign of $\frac{v^2}{2}-\gamma|u|^2$ in \eqref{dif} if $\gamma$ is positive. A natural idea to deal with this obstruction is to analyze the sign of the integrand according to $\frac{v^2}{2}-\gamma|u|^2>0$ and $\frac{v^2}{2}-\gamma|u|^2<0$.
More precisely, after introducing
\begin{equation*}
E_{1,t}=\{x:\frac{v^2(x,t)}{2}-\gamma|u|^2(x,t)>0\},\]\[E_{2,t}=\{x:\frac{v^2(x,t)}{2}-\gamma|u|^2(x,t)<0\},
\end{equation*}
one  constructs the functional
\begin{equation*}
J_{1}(t)=\frac{1}{\eta(t)} \int_{E_{1,t}\cup E_{2,t}} v\left(\chi_{E_{1,t}}(x)-\chi_{E_{2,t}}(x)\right) w\left(\frac{x}{\lambda_{1}(t)}\right)g\left(\frac{x}{\lambda_{2}(t)}\right) d x,
\end{equation*}
where $w(x)$, $g(x)$, $\lambda_1(t)$, $\lambda_2(t)$ and $\eta(t)$ are the same as in \cite{SIAM2021Linares}, i.e.,
\begin{equation*}
w(x)=\arctan (e^{x}),\hskip15pt g(x)=\frac{1}{2}{\rm sech}(x)=\frac{1}{e^{x}+e^{-x}},
\end{equation*}
\begin{equation*}
\lambda_{1}(t)=t^{p_{1}},\hskip10pt \lambda_{2}(t)=t^{p_1p_{2}}, \hskip10pt \eta(t)=t^{r_{1}},
\end{equation*}
(see \eqref{weight1}-\eqref{Equ2.4} for more details), and following the ideas and techniques in \cite{SIAM2021Linares}, one expects to get an identity of
$\frac{d}{dt}J_1(t)$ in the form
\begin{equation}\label{B1}
\begin{split}
\frac{1}{t} &\int_{\mathbb {R}} |\frac{v^2}{2}-\gamma|u|^2| w'\left(\frac{x}{\lambda_{1}(t)}\right)g\left(\frac{x}{\lambda_{2}(t)}\right) d x= \frac{d}{d t} J_{1}(t)+J_{1,int}(t),\quad\quad\forall \gamma\in \mathbb{R},\\
\end{split}
\end{equation} where $J_{1,int}(t)\in L^{1}(2,+\infty)$. Then integrating over $(2,+\infty)$ with respect to the time variable on both sides of the identity \eqref{B1}  and using the facts $\frac{1}{t}\notin L^{1}(2,+\infty)$ as well
as $\underset{t\in [1,+\infty)}{\sup}|J_1(t)|<\infty$, one shall obtain
\begin{equation*}
\lim _{t\rightarrow \infty} \inf \int_{| x | \lesssim t^{p}} |\frac{v^2}{2}-\gamma|u|^2| d x=0.
\end{equation*}
However, it might happen that $J_1(t)$ is not continuous in time which would imply that \eqref{B1} does not hold.

\medskip

To prove Theorem \ref{th1}, we introduce a new idea and techniques different from those in \cite{SIAM2021Linares}. We start mollifying the initial data $(u_0,v_0)\in (H^1(\mathbb{R}))^2$ through a smooth sequence $(u_{n0},v_{n0})\in (H^4(\mathbb{R}))^2$.  The well-posedness theory in \cite{TAMS2007Linares}  guarantees that the IVP associated to the Schr\"odinger-KdV system \eqref{Equ(0.1)} with  initial data $(u_{n0},v_{n0})\in (H^4(\mathbb{R}))^2$ and $\alpha\gamma>0$ admits a unique classical solution $(u_n,v_n)\in(C(\mathbb{R}^{+};H^4(\mathbb{R}))^2$. Then without constructing any functional but only using the coupled KdV structure from the  Schr\"odinger-KdV system as well as some properties of the weight functions $w(x)$ and $g(x)$, we successfully prove that

\begin{equation}\label{11step1res}
\int^{\infty}_2\int_{\mathbb{R}}\frac{1}{t}|\frac{1}{2}v^2_n-\gamma |u_n|^2 |w'(\frac{x}{\lambda_1(t)})g(\frac{x}{\lambda_2(t)})dxdt\lesssim C(\|u_{n0}\|_{H^1(\mathbb{R})},\|v_{n0}\|_{H^1(\mathbb{R})}).
\end{equation}

 Next, using the regularized solutions (see \eqref{s-13}-\eqref{s-15} for details)
 \begin{equation*}
 (u_{n},v_{n})\rightarrow (u,v) \text{ in } (C(0,+\infty; H^1(\mathbb{R})))^2 \quad\text{ as } \quad n\rightarrow +\infty,
 \end{equation*}
 \begin{equation*}
 \int^{\infty}_{2}\int_{\mathbb{R}}|v^2_n-v^2|\frac{w'(\frac{x}{\lambda_1(t)})g(\frac{x}{\lambda_2(t)})}{t}\,dxdt\rightarrow0 \quad\text{ as }\quad n\rightarrow +\infty,
 \end{equation*}
 \begin{equation*}
 \int^{\infty}_{2}\int_{\mathbb{R}}||u_n|^2-|u|^2|\frac{w'(\frac{x}{\lambda_1(t)})g(\frac{x}{\lambda_2(t)})}{t}\,dxdt\rightarrow0 \quad\text{ as }\quad n\rightarrow +\infty,
 \end{equation*}
 we have
 \begin{equation}\label{11result}
 \begin{split}
 &\int^{+\infty}_2\frac{1}{t}\int_{\mathbb{R}}|\frac{v^2}{2}-\gamma|u|^2|w'(\frac{x}{\lambda_1(t)})
 g(\frac{x}{\lambda_2(t)})\,dxdt<C(\|u_{0}\|_{H^1(\mathbb{R})},\|v_{0}\|_{H^1(\mathbb{R})}).
 \end{split}
\end{equation}
Finally, applying the fact $\frac{1}{t}\notin L^{1}(2,+\infty)$ and the properties of the weight functions $w(x)$ and $g(x)$, we obtain
\begin{equation*}
\lim _{t\rightarrow \infty} \inf \int_{| x | \lesssim t^{p}} |\frac{v^2}{2}-\gamma|u|^2| \,dx=0.
\end{equation*}
Utilizing similar techniques and the real and imaginary parts structure of the  coupled Schr\"odinger equation from the Schr\"odinger-KdV system, we also have
\begin{equation*}
\lim _{t\rightarrow \infty} \inf \int_{| x | \lesssim t^{p}} |u(\alpha v+\beta|u|^2)|\, dx=0.
\end{equation*}

\begin{remark}
We notice that we do not need the smallness assumption as that for the generalized KdV equations \cite{CMP2019Ponce}
because of three conserved laws for the Schr\"odinger-KdV system \eqref{Equ(0.1)}
with $\alpha\gamma>0$ and the initial data $(u_0,v_0)\in (H^1(\mathbb{R}))^2$, which was first observed in \cite{TAMS2007Linares}.
\end{remark}

Observing that the local energy decay criterion in Theorem \ref{th1} is lower-order and ``mixed", the second goal in this paper is to
get the first-order and ``unmixed'' local energy decay criterion for the Schr\"odinger-KdV system \eqref{Equ(0.1)} with the resonant
interaction  if $\beta=0$ and with no-resonant interactions if $\beta>0$.

\medskip

\begin{theorem}\label{th2} \textbf{(non-resonant interactions and without smallness assumption)} Let $\alpha$, $\gamma\in \mathbb{R}^{+}$ and $\beta> 0$. Assume that $(u,v)$ is the solution to the Schr\"odinger-KdV system \eqref{Equ(0.1)} with $(u_0,v_0)\in (H^1(\mathbb{R}))^2$ such that $(u,v)\in (C(\mathbb{R}^{+};H^1(\mathbb{R})))^2$. Then
\begin{equation}\label{decay2}
\begin{split}
\lim _{t\rightarrow \infty} \inf \int_{| x | \lesssim t^{p}} (\partial_xv)^2 d x=\lim _{t\rightarrow \infty} \inf \int_{| x | \lesssim t^{p}} |\partial_xu|^2 d x=0,\quad\quad \forall p\in (0,\frac{2}{3}).
\end{split}
\end{equation} Moreover, for any $k\in (2,\infty)$, we have
\begin{equation}\label{decay3}
\begin{split}\lim _{t\rightarrow \infty} \inf \int_{| x | \lesssim t^{p}} |v|^{k} d x=\lim _{t\rightarrow \infty} \inf \int_{| x | \lesssim t^{p}} |u|^{k} d x=0.\end{split}
\end{equation}
\end{theorem}

\medskip

\begin{remark} \label{rem1} From the proof of \eqref{decay2} in Section 4, we observe that \eqref{decay2} still holds  under the same assumptions in Theorem \ref{th2} but with the restriction $\beta=0$.
\end{remark}

To show Theorem \ref{th2}, we use different techniques from those in \cite{SIAM2021Linares}. The authors in \cite{SIAM2021Linares}  directly prove \eqref{decay2} by introducing the following functional with  one weight function (see (4.1) in \cite{SIAM2021Linares})
\begin{equation*}
\frac{\theta}{2\eta(t)}\int_{\mathbb{R}}v^2\psi_{l}(\frac{x}{\lambda(t)})dx
+\frac{\mu}{\eta(t)}{\Im}\int_{\mathbb{R}}u\overline{\partial_xu}\psi_{l}(\frac{x}{\lambda(t)})dx,
\end{equation*}
where $\theta$, $\mu$, $l\in \mathbb{R}\setminus\{0\}$, $\psi_l(x)=\frac{2l}{\pi}\arctan(e^{\frac{x}{l}})$, $\lambda(t)=\frac{t^{p_1}}{\ln^{q_1}(t)}$, $\Im $ stands for the imaginary part of a complex valued function, and by using the  boundedness of the weighted energy, that is,
\begin{equation}\label{decay4}
\int^{\infty}_1\frac{1}{t\ln t}\int_{\mathbb{R}}(v^2+|u|^2)\psi^{'}_{l}(\frac{x}{t^{p_1}})dxdt<\infty.
\end{equation}
But the key energy inequality \eqref{decay4} is missing if $\alpha,\gamma\in \mathbb{R}^{+}$.

\medskip

We introduce the following well-defined  functionals with two weight functions
\begin{equation*}
J_2(t)=\frac{\theta_2}{\eta(t)}\int_{\mathbb{R}}v^2w(\frac{x}{\lambda_1(t)})g(\frac{x}{\lambda_2(t)})dx,\quad \quad \theta_2\in \mathbb{R}^{+},
\end{equation*}
(see Proposition \ref{prop2} below) and
\begin{equation*}
 J_3(t)=\frac{\theta_3}{\eta(t)}\int_{\mathbb{R}}{\Im}[u(x,t)\partial_x\overline{u}(x,t)]w(\frac{x}{\lambda_1(t)})g(\frac{x}{\lambda_2(t)})dx,\quad \theta_3\in \mathbb{R}^{+},
\end{equation*}
(see Proposition \ref{prop3} below) and construct the crucial identity
 \begin{equation}\label{cru1}
 \begin{split}
&\frac{3\theta_2}{t} \int_{\mathbb {R}}|\partial_xv|^2 w'(\frac{x}{\lambda_1(t)})g(\frac{x}{\lambda_2(t)}) d x+\frac{2\theta_3}{t} \int_{\mathbb {R}}|\partial_xu|^2 w'(\frac{x}{\lambda_1(t)})g(\frac{x}{\lambda_2(t)}) d x \\
=&-\frac{d}{dt}(J_2(t)+J_3(t))+J_{2,int}(t)+J_{3,int}(t)+\frac{2\theta_2}{t}\int_{\mathbb{R}}(\frac{v^3}{3}-\gamma |u|^2v)w'(\frac{x}{\lambda_1(t)})g(\frac{x}{\lambda_2(t)})dx\\
&-\frac{\beta\theta_3}{2t}\int_{\mathbb{R}}|u|^4w'(\frac{x}{\lambda_1(t)})g(\frac{x}{\lambda_2(t)})dx+(-\frac{2\theta_2\gamma}{\eta(t)}+\frac{\theta_3\alpha}{\eta(t)})\int_{\mathbb{R}} |u|^2\partial_xvw(\frac{x}{\lambda_1(t)})g(\frac{x}{\lambda_2(t)})dx,
\end{split}
\end{equation}
where $J_{2,int}(t)$, $J_{3,int}(t)\in L^1(2,+\infty)$ and $\beta\in \mathbb{R}$.

Then choosing
\begin{equation*}
\theta_2>0,\quad \theta_3=\frac{2\theta_2\gamma}{\alpha},\quad \alpha>0,\quad \gamma>0,\quad \beta\geq0,
\end{equation*}
and using
\begin{equation*}
\int^{+\infty}_2\frac{1}{t}\int_{\mathbb{R}}|\frac{v^2}{2}-\gamma|u|^2|w'(\frac{x}{\lambda_1(t)})g(\frac{x}{\lambda_2(t)})dxdt<C(\|u_{0}\|_{H^1(\mathbb{R})},\|v_{0}\|_{H^1(\mathbb{R})}),
\end{equation*}
(see Proposition \ref{prop1} below), and
\begin{equation*}
\int^{+\infty}_2\frac{1}{t} \int_{\mathbb {R}}| u(\alpha v+\beta |u|^2)| w'\left(\frac{x}{\lambda_{1}(t)}\right)g\left(\frac{x}{\lambda_{2}(t)}\right) d xdt\leq C(\|u_0\|_{H^1(\mathbb{R})},\|v_0\|_{H^1(\mathbb{R})}),
\end{equation*}
see Corollary \ref{prop4})\footnote{Proposition \ref{prop1} and Corollary \ref{prop4} will be used to deal with the  term $\frac{2\theta_2}{t}\int_{\mathbb{R}}(\frac{v^3}{3}-\gamma |u|^2v)w'(\frac{x}{\lambda_1(t)})g(\frac{x}{\lambda_2(t)})dx$ in \eqref{cru1}.}
the crucial identity \eqref{cru1} implies  that

\begin{equation}\label{cru2}
\begin{split}
&\int^{+\infty}_{2}\frac{1}{t} \int_{\mathbb {R}}\Big(3|\partial_xv|^2+\frac{4\gamma}{\alpha}|\partial_xu|^2\Big) w'(\frac{x}{\lambda_1(t)})g(\frac{x}{\lambda_2(t)}) dxdt \\
&\hskip35pt+\frac{\beta\gamma}{3\alpha}\int^{+\infty}_{2}\frac{1}{t}\int_{\mathbb{R}}|u|^4w'(\frac{x}{\lambda_1(t)})g(\frac{x}{\lambda_2(t)})dxdt<+\infty.
\end{split}
\end{equation}
After using the properties of the weight functions $w(x)$ and $g(x)$, we immediately get \eqref{decay2} from \eqref{cru2}.

With \eqref{cru2} at hand, the restriction $\beta>0$ and the key observation that
\begin{equation*}
\int^{\infty}_{2}\frac{1}{t}\int_{\mathbb{R}}|u|^{2+m} w'(\frac{x}{\lambda_1(t)})g(\frac{x}{\lambda_2(t)})\, dxdt <\infty\hskip5pt \Leftrightarrow \hskip5pt \int^{\infty}_{2}\frac{1}{t}\int_{\mathbb{R}}|v|^{2+m} w'(\frac{x}{\lambda_1(t)})g(\frac{x}{\lambda_2(t)})\, dxdt <\infty,
\end{equation*}
 for any given $m\in \mathbb{R}^{+}$, we can prove \eqref{decay3}.

\medskip

We observe that in Theorem \ref{th2} we require $\beta\in \mathbb{R}^{+}$ and in the
Remark \ref{rem1} we must have $\beta=0$. Our next result considers
the case $\beta\in \mathbb{R}^{-}$. First, we will describe a function depending on the initial data
$(\|u_0\|_{H^1}, \|u_0\|_{H^1})$ which is crucial in our analysis.
\begin{equation}\label{exp1-a}
\Phi \left(\left\|u_{0}\right\|_{H^{1}},\left\|v_{0}\right\|_{H^{1}}\right):=C^{\frac{1}{2}}_{\alpha,\beta,\gamma}
(\|u_0\|_{H^1}+\|v_0\|_{H^1}+\|u_0\|^{5}_{H^1}+\|v_0\|^{5}_{H^1}),
\end{equation}
where
\begin{equation}\label{exp1-b}
\begin{split}
C_{\alpha,\beta,\gamma}= &2(1+\frac{|\gamma|}{|\alpha|}+\frac{4\gamma^2}{|\alpha|^2})+\frac{4(C|\alpha\gamma|+2|\alpha|+3|\gamma|+\frac{32\gamma^2}{\min\{|\gamma|,\frac{|\alpha|}{2}\}})+C|\beta\gamma|}{\min\{|\gamma|,\frac{|\alpha|}{2}\}}\\
 +\frac{(|\alpha\gamma^2|+\frac{|\beta\gamma|}{2})^2}{(\min\{|\gamma|,\frac{|\alpha|}{2}\})^2}&C+\frac{C}{(\min\{|\gamma|,\frac{|\alpha|}{2}\})^{\frac{4}{3}}|\alpha|^{\frac{1}{3}}}\Big[(|\alpha|+2|\gamma|)^{\frac{5}{3}}+\frac{|\gamma|^{10}}{(\min\{|\gamma|,\frac{|\alpha|}{2}\})^{\frac{20}{3}}|\alpha|^{\frac{5}{3}}}\Big].
 \end{split}
 \end{equation}
 Here the constant $C$ depends on the best constants for the Gagliardo-Nirenberg inequalities that we have used to
 obtain the expression above but not on the parameters present in the Schr\"odinger-KdV system.
 The construction of this function is given in \eqref{Equ2.2}, \eqref{exp1}, \eqref{phi} below. With this information on hand, we have Theorem \ref{th3}.

\begin{theorem}\label{th3} \textbf{(Non-resonant interactions and with smallness assumption)}
Let $\alpha$, $\gamma\in \mathbb{R}^{+}$ and $\beta\in \mathbb{R}^{-}$. Assume that $(u,v)$ is the solution to the Schr\"odinger-KdV system \eqref{Equ(0.1)} with $(u_0,v_0)\in (H^1(\mathbb{R}))^2$ such that
$(u,v)\in (C(\mathbb{R}^{+};H^1(\mathbb{R})))^2$. If the initial data $(u_0,v_0)\in (H^1(\mathbb{R}))^2$ satisfies
\begin{equation}\label{smal}
-\beta \Phi\left(\left\|u_{0}\right\|_{H^{1}},\left\|v_{0}\right\|_{H^{1}}\right)\leq \alpha \gamma,
\end{equation}
where the function $\Phi(\cdot,\cdot)$ is given in \eqref{exp1-a},
then we obtain the local energy decay,
\begin{equation}\label{decaysmal2}
\begin{split}
\lim _{t\rightarrow \infty} \inf \int_{| x | \lesssim t^{p}} (\partial_xv)^2 d x=\lim _{t\rightarrow \infty} \inf \int_{| x | \lesssim t^{p}} |\partial_xu|^2 d x=0,\quad\quad \forall p\in (0,\frac{2}{3}).
\end{split}
\end{equation} Moreover, for any given $k\in (2,+\infty)$, we have
\begin{equation}\label{decaysmal3}
\begin{split}\lim _{t\rightarrow \infty} \inf \int_{| x | \lesssim t^{p}} |v|^{k} d x=\lim _{t\rightarrow \infty} \inf \int_{| x | \lesssim t^{p}} |u|^{k} d x=0.
\end{split}
\end{equation}
\end{theorem}

\medskip

To prove Theorem \ref{th3}, we first need to analyze the $L^1(2,+\infty)$ integrability of the function
\begin{equation*}
-\frac{\beta\theta_3}{2t}\int_{\mathbb{R}}|u|^4w'(\frac{x}{\lambda_1(t)})g(\frac{x}{\lambda_2(t)})dx,
\end{equation*}
which comes up in the right hand side of the identity \eqref{cru1}, for $\beta<0$. Again, by using Proposition \ref{prop1}, Corollary \ref{prop4}, Lemma \ref{result6} and the restriction $\beta<0$, we can prove
\begin{equation}\label{a1}\int^{+\infty}_2\int_{\mathbb{R}}\frac{1}{t}|u|^2|v||\alpha\gamma +\frac{v\beta}{2}|w'(\frac{x}{\lambda_1(t)})g(\frac{x}{\lambda_2(t)})dxdt<C(\|u_{0}\|_{H^1(\mathbb{R})},\|v_{0}\|_{H^1(\mathbb{R})}).\end{equation}
Using the \lq\lq smallness" restriction \eqref{smal}, we deduce from \eqref{a1} that
\begin{equation*}
\int^{+\infty}_2\int_{\mathbb{R}}\frac{1}{t}|u|^2 |\alpha v|w'(\frac{x}{\lambda_1(t)})g(\frac{x}{\lambda_2(t)})dxdt<C(\|u_{0}\|_{H^1(\mathbb{R})},\|v_{0}\|_{H^1(\mathbb{R})}).
\end{equation*}
This estimate will help us to show that
\begin{equation}\label{a0}
-\frac{\beta\theta_3}{2t}\int_{\mathbb{R}}|u|^4w'(\frac{x}{\lambda_1(t)})g(\frac{x}{\lambda_2(t)})dx\in L^1(2,+\infty).\end{equation}
Once we obtain \eqref{a0}, we will utilize the identity \eqref{cru1} and a similar argument as the one employed in the proof Theorem \ref{th2} to establish Theorem \ref{th3}.

\begin{remark}
In comparison to the case $\beta\geq 0$, we have added the \lq\lq smallness" restriction \eqref{smal} for the case
$\beta<0$.  The restriction \eqref{smal} implies that we have the local decay \eqref{decaysmal2} and \eqref{decaysmal3} in the following situations: for any large initial data $(u_0,v_0)\in (H^1(\mathbb{R}))^2$ and small absolute value of $\beta<0$ and for small initial data $(u_0,v_0)\in (H^1(\mathbb{R}))^2$  and large
absolute value for the parameter $\beta<0$.
\end{remark}

\indent Noticing that the local energy decay region $\{x:| x | \lesssim t^{p}\}$ in Theorems \ref{th1}, \ref{th2} and \ref{th3} is a symmetric  interval  about origin. As a corollary, we can consider the non-centered case, that is, we will translate the origin to the right by $t^{m}$ units.

\begin{corollary}\label{result2}
Under the same assumptions as in Theorems \ref{th2} and \ref{th3}, we have
\begin{equation*}
\begin{split}
\liminf _{t \rightarrow\infty} \int_{| x-t^{m} | \lesssim t^{p}}|u|^{k}(x, t)  \mathrm{~d} x=\liminf _{t \rightarrow \infty} \int_{| x-t^{m} | \lesssim t^{p}}|v|^{k}(x, t) \mathrm{~d} x=0, \quad k\in (2,+\infty)
\end{split}
\end{equation*}
and
\begin{equation*}
\begin{split}
\lim _{t\rightarrow \infty} \inf \int_{| x-t^{m} | \lesssim t^{p}} (\partial_xv)^2 d x=\lim _{t\rightarrow \infty} \inf \int_{|  x-t^{m} | \lesssim t^{p}} |\partial_xu|^2 d x=0,
\end{split}
\end{equation*}
where $0<p<\frac{2}{3}$ and $0<m<1-\frac{p}{2}$.
\end{corollary}

\begin{remark} Although Theorems \ref{th2} and \ref{th3} can be extended to the non-centered case, we have the same restrictions on the index $m$. The idea and techniques in the proof of Corollary \ref{result2} are similar to those one in  Theorems \ref{th2} and \ref{th3}. The main difference in the argument is to  redefine  the  corresponding functionals  by shifting the spatial variable of the weight function to the right by $t^{m}$ units. Hence, we omit the proof of Corollary \ref{result2}.
\end{remark}

\begin{remark} We remark that our results depend on the global in time solution to the Schr\"odinger-KdV system \eqref{Equ(0.1)} on the line\footnote{One may refer to the other  well-posedness results  of the Schr\"odinger-KdV system, for example \cite{JDE2006Corcho,Arxiv2024Chenjie,SIAM2024SOS}.}.  Corcho-Linares \cite{TAMS2007Linares} proved the global well-posedness (GWP) of the Schr\"odinger-KdV system \eqref{Equ(0.1)} in $H^1(\mathbb{R})\times H^1(\mathbb{R})$ in 2007 by the mass and the energy conservation laws, Pecher \cite{DIE2005Pecher} extended the GWP in $H^k(\mathbb{R})\times H^s(\mathbb{R})$ with $k=s\geq \frac{3}{5}$, Wu \cite{DIE2010Wuyifei} extended the GWP in $H^k(\mathbb{R})\times H^s(\mathbb{R})$ with $k=s> \frac{1}{2}$.  Most recently, Correia-Linares-Silva \cite{Arxiv2024Linares} have proved the GWP in $H^k(\mathbb{R})\times H^s(\mathbb{R})$ with\footnote{For the most recent local well-posedness result for the Schr\"odinger-KdV system, one may refer to \cite{Arxiv2024Ban} and \cite{Arxiv2024Linares}. }
\begin{equation}\label{param1}
(k,s)\in \mathcal{A}\cap \{(k,s):k>\frac{1}{2}, s>\frac{1}{2}\},
\end{equation}
where $\mathcal{A}:=\{(k,s)\in \mathbb{R}^2: k\geq 0, s>-\frac{3}{4}, s<4k, -2<k-s<3\}$. In this paper, we consider the initial data in $H^1(\mathbb{R})\times H^1(\mathbb{R})$. Generalizing the results in Theorems \ref{th2} and \ref{th3} to that of  the Schr\"odinger-KdV system \eqref{Equ(0.1)} when the initial data in $H^k(\mathbb{R})\times H^s(\mathbb{R})$ with $(k,s)$ satisfying \eqref{param1} is not an easy problem. It might be possible that the ideas and some technical modifications in our arguments can be applied to deal with the case \eqref{param1}.
\end{remark}

\medskip

The structure of this paper is as follows. In Section 2, we derive some  {\it a priori} estimates for the solution of the Schr\"odinger-KdV system. We also define two useful functionals which we will employ in our arguments. In Section 3, we prove several key estimates as well as two identities. In Section 4, we prove Theorems \ref{th1} and \ref{th2}. In the last section, we show Theorem \ref{th3}.

\section{Preparation}

\hspace*{\parindent} In this section, we will use an {\it a priori} estimate to the solutions to the Schr\"odinger-KdV system \eqref{Equ(0.1)} to construct two new Virial functionals. As observed and proved in \cite{TAMS2007Linares}, the solution $(u,v)$ to the Schr\"odinger-KdV system \eqref{Equ(0.1)} has the following {\it a priori} estimates.

\begin{lemma}\label{result6} (see \cite{TAMS2007Linares}) Let $\alpha,\beta,\gamma\in \mathbb{R}$ satisfying $\alpha \cdot \gamma>0$. The Schr\"odinger-KdV system \eqref{Equ(0.1)} with $(u_0,v_0)\in (H^1(\mathbb{R}))^2$ admits a unique solution $(u,v)\in (C([0,\infty); H^1(\mathbb{R})))^2$.  Moreover,
\begin{equation}\label{Equ2.2}
\begin{split}
\sup _{t\in[0,+\infty)}\left(\|u(t)\|_{H^{1}}+\|v(t)\|_{H^{1}}\right) \leq \Phi \left(\left\|u_{0}\right\|_{H^{1}},\left\|v_{0}\right\|_{H^{1}}\right),
\end{split}
\end{equation}
where $\Phi$ is a function only depending on $\left\|u_{0}\right\|_{H^{1}}$ and $\left\|v_{0}\right\|_{H^{1}}.$
\end{lemma}

Next, we aim to  construct some weighted  functionals.  We first introduce two weight functions. Let
 \begin{equation}\label{weight1}
 g(x)=\frac{1}{e^{x}+e^{-x}},\quad \text{then} \quad g(x) \backsim e^{-|x|}.
 \end{equation}
It can be directly checked that one of the primitive functions of $g(x)$ is
\begin{equation}\label{weight2}
w(x):=\int g(x)dx=\arctan (e^{x}).
\end{equation}
Moreover
\begin{equation}\label{est1}
|w''(x)|+|w'''(x)|\lesssim e^{-|x|}.
\end{equation}
For  any $ p_{1}>0, p_2>1, r_{1}>0$, we set
\begin{equation}\label{Equ2.4}
\begin{split}
\lambda_{1}(t)=t^{p_{1}},\quad \lambda_{2}(t)=t^{p_1p_{2}},  \text{ and }  \eta(t)=t^{r_{1}}.
\end{split}
\end{equation}
In the next section, we will impose some restrictions on the indices $p_1$, $p_2$ and $r_1$.

Now, for the solution $(u, v)$  of \eqref{Equ(0.1)} with $(u_0,v_0)\in (H^1(\mathbb{R}))^2$,  we  define two
functionals  $\{J_i(t)\}_{i=2,3}$ as below,
\begin{equation}\label{J2}
J_2(t):=\frac{\theta_2}{\eta(t)}\int_{\mathbb{R}}v^2(x,t)\,w(\frac{x}{\lambda_1(t)})g(\frac{x}{\lambda_2(t)})\,dx,\quad\theta_2\in \mathbb{R}^{+},
\end{equation}
and
\begin{equation}\label{J3}
J_3(t):=\frac{\theta_3}{\eta(t)}\int_{\mathbb{R}}{\Im }[u(x,t)\overline{\partial_xu(x,t)}]w(\frac{x}{\lambda_1(t)})g(\frac{x}{\lambda_2(t)})dx,\quad \theta_3\in \mathbb{R}^{+},
\end{equation}
where $\Im u$ denotes the imaginary part of the complex valued function $u$.

Applying Lemma \ref{result6}, one can directly check that $J_2(t)$ and $J_3(t)$
are well-defined. Moreover, $J_2(t), J_3(t)\in C(\mathbb{R}^{+})$ for any $(u,v)\in (C([0,\infty); H^1(\mathbb{R})))^2$.

\section{Some key estimates and useful identities}
\hspace*{\parindent} In this section, we aim to derive several estimates for
\begin{equation*}
\int^{+\infty}_{2}\int_{\mathbb{R}}\frac{1}{t}|\frac{v^2}{2}-\gamma|u|^2|w'(\frac{x}{\lambda_1(t)})g(\frac{x}{\lambda_2(t)})dxdt
\end{equation*}
and
\begin{equation*}
\int^{+\infty}_{2}\int_{\mathbb{R}}\frac{1}{t}| u(\alpha v+\beta |u|^2)| w'\left(\frac{x}{\lambda_{1}(t)}\right)g\left(\frac{x}{\lambda_{2}(t)}\right)dxdt.
\end{equation*}
Also we will construct two  identities of the functional $\frac{d}{dt}J_{j}(t)$  by finding another functional $J_{j,int}(t)\in L^1(2,+\infty)$ satisfying
\begin{equation*}
\frac{d}{dt}J_{j}(t)=-J_{j,int}(t)+\widetilde{J}_{j}(t),\quad\quad j=2,3.
\end{equation*}
These estimates and identities will be used to prove the lower-order ``mixed" local energy decay and first-order ``unmixed" local energy decay in next section.

\begin{proposition}\label{prop1}For any given $(\alpha,\beta,\gamma)\in \mathbb{R}^3$ with $\alpha\gamma>0$, let $(u,v)$ be the global weak solution to the Schr\"odinger-KdV system \eqref{Equ(0.1)} with the initial data $(u_0,v_0)\in (H^1(\mathbb{R}))^2$, then
\begin{equation}\label{1result}
\int^{+\infty}_2\frac{1}{t}\int_{\mathbb{R}}|\frac{v^2}{2}-\gamma|u|^2|w'(\frac{x}{\lambda_1(t)})g(\frac{x}{\lambda_2(t)})dxdt < C(\|u_{0}\|_{H^1(\mathbb{R})},\|v_{0}\|_{H^1(\mathbb{R})}),
\end{equation}
where the functions $g(x)$, $w(x)$, $\lambda_{1}(t)$ and $\lambda_{2}(t)$  are introduced in
\eqref{weight1}-\eqref{Equ2.4}. More precisely,
\begin{equation*}
\begin{split}
g(x)=\frac{1}{e^{x}+e^{-x}},\quad w(x)=\int g(x)dx,\\
\lambda_{1}(t)=t^{p_{1}},\hskip15pt\lambda_{2}(t)=t^{p_1p_{2}},
\end{split}
\end{equation*}
with  $ r_{1}>0$, $p_1+r_1=1$, $p_2>1$ and $0<p_1<\frac{2}{p_2+2}$.
\end{proposition}

\noindent{\it Proof.} We prove Proposition \ref{prop1} in three steps.

\medskip

\textbf{\textbf{Step 1:  Suppose that $(u_n,v_n)$ is a smooth solution of  \eqref{Equ(0.1)} with the initial data $(u_{n0},v_{n0})\in (H^s(\mathbb{R}))^2$} where $s>4$.}  Then, we can rewrite the second equation of the Schr\"odinger-KdV system
\begin{equation}\label{smooth-sol}
\left\{
\begin{array}{llll}i\partial_tu_n+\partial^2_xu_n=\alpha u_nv_n+\beta u_n|u_n|^2,\hskip25pt (x,t)\in \mathbb{R}\times \mathbb{R}^{+},\\
\partial_tv_n+\partial^3_xv_n+v_n\partial_xv_n=\gamma \partial_x(|u_n|^2), \quad\quad (x,t)\in \mathbb{R}\times \mathbb{R}^{+},\\
\left.(u_n, v_n)\right|_{t=0} =\left(u_{n0}, v_{n0}\right),\\
\end{array}
\right.
\end{equation}
as
\begin{equation}\label{s-1}
\partial_tv_n+\partial^3_xv_n+\partial_x(\frac{1}{2}v^2_n-\gamma |u_n|^2 )=0,
\end{equation}
which holds pointwise.

We set
\begin{equation*}
\begin{split}
&E_{-,t}=\{x:\frac{v^2_n}{2}(x,t)-\gamma|u_n|^2(x,t)<0\},\\
&E_{0,t}=\{x:\frac{v^2_n}{2}(x,t)-\gamma|u_n|^2(x,t)=0\},\\
&E_{+,t}=\{x:\frac{v^2_n}{2}(x,t)-\gamma|u_n|^2(x,t)>0\}.
\end{split}
\end{equation*}
It can be directly checked that $E_{-,t}\subset\mathbb{R} $ and $E_{+,t}\subset\mathbb{R}$ are open for any fixed $t\in \mathbb{R}^{+}$ by using the local sign preserving  properties of continuous functions and the fact that $\frac{v^2_n}{2}(x,t)-\gamma|u_n|^2(x,t)\in C(\mathbb{R})$ for any fixed $t\in \mathbb{R}^{+}$. The case $E_{0,t}$ means that the coupled KdV equation in the Schr\"odinger-KdV \eqref{smooth-sol} is linear at the time $t$.

Multiplying both sides of \eqref{s-1} by the multiplier $\chi_{E_{+,t}}(x)\frac{1}{\eta(t)}w(\frac{x}{\lambda_1(t)})g(\frac{x}{\lambda_2(t)})$, where $\chi_{E_{+,t}}(x)$ is the characteristic function defined on the set $E_{+,t}$, and $\eta(t)$ is defined in \eqref{Equ2.4}, integrating over $\mathbb{R}$ with respect to the space variable, and then integrating over $(2,+\infty)$ with respect to the time variable, we arrive at
\begin{equation}\label{s-2}
\begin{split}
&\underbrace{\int^{+\infty}_2\int_{\mathbb{R}}\partial_tv_n\chi_{E_{+,t}}(x)\frac{1}{\eta(t)}w(\frac{x}{\lambda_1(t)})g(\frac{x}{\lambda_2(t)})dxdt}_{I}\\
&+\underbrace{\int^{+\infty}_2\int_{\mathbb{R}}\partial^3_xv_n\chi_{E_{+,t}}(x)\frac{1}{\eta(t)}w(\frac{x}{\lambda_1(t)})g(\frac{x}{\lambda_2(t)})dxdt}_{I\!I}\\
&\underbrace{+\int^{+\infty}_2\int_{\mathbb{R}}\partial_x(\frac{1}{2}v^2_n-\gamma |u_n|^2 )\chi_{E_{+,t}}(x)\frac{1}{\eta(t)}w(\frac{x}{\lambda_1(t)})g(\frac{x}{\lambda_2(t)})dxdt}_{I\!I\!I}=0.
\end{split}
\end{equation}
Now, we move to estimate $I$, $II$, and $III$.

For $III$: Using integration by parts, the fact \[\frac{\partial \chi_{E_{+,t}(x)}}{\partial_x}|_{x\in E_{+,t}\cup E_{-,t}}=0,\] and the boundary of $E_{+,t}$ is $E_{0,t}$, we have

\begin{equation}\label{s-3}
\begin{split}
 I\!I\!I=&\int_{\mathbb{R}}\partial_x(\frac{1}{2}v^2_n-\gamma |u_n|^2 )\chi_{E_{+,t}}(x)\frac{1}{\eta(t)}w(\frac{x}{\lambda_1(t)})g(\frac{x}{\lambda_2(t)})dx\\
=&\int_{E_{-,t}\cup E_{0,t}\cup E_{+,t} }\partial_x(\frac{1}{2}v^2_n-\gamma |u_n|^2 )\chi_{E_{+,t}}(x)\frac{1}{\eta(t)}w(\frac{x}{\lambda_1(t)})g(\frac{x}{\lambda_2(t)})dx
\end{split}
\end{equation}
\begin{equation*}
\begin{split}
=&\int_{E_{+,t} }\partial_x(\frac{1}{2}v^2_n-\gamma |u_n|^2 )\chi_{E_{+,t}}(x)\frac{1}{\eta(t)}w(\frac{x}{\lambda_1(t)})g(\frac{x}{\lambda_2(t)})dx\\
=&-\int_{E_{+,t} }(\frac{1}{2}v^2_n-\gamma |u_n|^2 )\cancel{\partial_x\chi_{E_{+,t}}(x)}\frac{1}{\eta(t)}w(\frac{x}{\lambda_1(t)})g(\frac{x}{\lambda_2(t)})dx\\
&-\int_{E_{+,t} }(\frac{1}{2}v^2_n-\gamma |u_n|^2 )\chi_{E_{+,t}}(x)\partial_x[\frac{1}{\eta(t)}w(\frac{x}{\lambda_1(t)})g(\frac{x}{\lambda_2(t)})]dx\\
&+(\frac{1}{2}v^2_n-\gamma |u_n|^2 )\cancel{\chi_{E_{+,t}}(x)}\frac{1}{\eta(t)}w(\frac{x}{\lambda_1(t)})g(\frac{x}{\lambda_2(t)})|_{E_{0,t}}\\
=&-\int_{E_{+,t} }(\frac{1}{2}v^2_n-\gamma|u_n|^2) \chi_{E_{+,t}}(x)\partial_x[\frac{1}{\eta(t)}w(\frac{x}{\lambda_1(t)})g(\frac{x}{\lambda_2(t)})]dx.
\end{split}
\end{equation*}
Therefore,
\begin{equation}\label{s-4}
\begin{split}
-I\!I\!I=&\int^{\infty}_2\int_{E_{+,t} }(\frac{1}{2}v^2_n-\gamma |u_n|^2 )\chi_{E_{+,t}}(x)\partial_x[\frac{1}{\eta(t)}w(\frac{x}{\lambda_1(t)})g(\frac{x}{\lambda_2(t)})]dxdt\\
=&\int^{\infty}_2\frac{1}{\eta(t)\lambda_1(t)}\int_{E_{+,t} }(\frac{1}{2}v^2_n-\gamma |u_n|^2 )\chi_{E_{+,t}}(x)w'(\frac{x}{\lambda_1(t)})g(\frac{x}{\lambda_2(t)})dxdt\\
&+\underbrace{\int^{\infty}_2\frac{1}{\eta(t)\lambda_2(t)}\int_{E_{+,t} }(\frac{1}{2}v^2_n-\gamma |u_n|^2 )\chi_{E_{+,t}}(x)w(\frac{x}{\lambda_1(t)})g'(\frac{x}{\lambda_2(t)})dxdt}_{I\!I\!I_a},
\end{split}
\end{equation}
where
\begin{equation}\label{s-5}
|\!I\!I\!I_a|<C(\|u_{n0}\|_{H^1(\mathbb{R})},\|v_{n0}\|_{H^1(\mathbb{R})} ).
\end{equation}

In fact, using the restriction $r_1+p_1=1,\, p_1>0,\,p_2>1$,  one has
\begin{equation*}
\begin{split}
|I\!I\!I_a| &\lesssim \int^{\infty}_2\frac{1}{\eta(t)\lambda_2(t)}\int_{\mathbb{R} }|\frac{1}{2}v^2_n-\gamma |u_n|^2| dxdt
 \lesssim \sup_{t\in (2,+\infty)}(\|u_n\|_{L^2_x}+\|v_n\|_{L^2_x})\int^{\infty}_2\frac{1}{\eta(t)\lambda_2(t)}dt\\
  &\lesssim C(\|u_{n0}\|_{H^1_x},\|v_{n0}\|_{H^1_x})\int^{\infty}_2\frac{1}{t^{r_1+p_1p_2}}dt
  \lesssim C(\|u_{n0}\|_{H^1},\|v_{n0}\|_{H^1}).
\end{split}
\end{equation*}

For $II$: Using integration by parts twice, the fact \[\frac{\partial \chi_{E_{+,t}(x)}}{\partial_x}|_{x\in E_{+,t}\cup E_{-,t}}=0,\] and the boundary of $E_{+,t}$ is $E_{0,t}$, we have
\begin{equation*}
\begin{split}
I\!I=&\int_{\mathbb{R}}\partial^3_xv_n\chi_{E_{+,t}}(x)\frac{1}{\eta(t)}w(\frac{x}{\lambda_1(t)})g(\frac{x}{\lambda_2(t)})dx\\
=&\int_{E_{-,t}\cup E_{0,t}\cup E_{+,t}}\partial^3_xv_n\chi_{E_{+,t}}(x)\frac{1}{\eta(t)}w(\frac{x}{\lambda_1(t)})g(\frac{x}{\lambda_2(t)})dx\\
=&\int_{E_{+,t}}\partial^3_xv_n\chi_{E_{+,t}}(x)\frac{1}{\eta(t)}w(\frac{x}{\lambda_1(t)})g(\frac{x}{\lambda_2(t)})dx\\
=&-\int_{E_{+,t}}\partial^2_xv_n\cancel{\partial_x{\chi_{E_{+,t}}(x)}}\frac{1}{\eta(t)}w(\frac{x}{\lambda_1(t)})g(\frac{x}{\lambda_2(t)})dx
\end{split}
\end{equation*}
\begin{equation*}
\begin{split}
&-\int_{E_{+,t}}\partial^2_xv_n\chi_{E_{+,t}}(x)\frac{1}{\eta(t)}\partial_x[w(\frac{x}{\lambda_1(t)})g(\frac{x}{\lambda_2(t)})]dx\\
&+\partial^2_xv_n\cancel{\chi_{E_{+,t}}(x)}\frac{1}{\eta(t)}w(\frac{x}{\lambda_1(t)})g(\frac{x}{\lambda_2(t)})|_{E_{0,t}}\\
=&-\int_{E_{+,t}}\partial^2_xv_n\chi_{E_{+,t}}(x)\frac{1}{\eta(t)}\partial_x[w(\frac{x}{\lambda_1(t)})g(\frac{x}{\lambda_2(t)})]dx\\
=&\int_{E_{+,t}}\partial_xv_n\cancel{\partial_x(\chi_{E_{+,t}}(x))}\frac{1}{\eta(t)}\partial_x[w(\frac{x}{\lambda_1(t)})g(\frac{x}{\lambda_2(t)})]dx\\
&+\int_{E_{+,t}}\partial_xv_n\chi_{E_{+,t}}(x)\frac{1}{\eta(t)}\partial^2_x[w(\frac{x}{\lambda_1(t)})g(\frac{x}{\lambda_2(t)})]dx\\
&+\partial_xv_n\cancel{\chi_{E_{+,t}}(x)}\frac{1}{\eta(t)}\partial_x[w(\frac{x}{\lambda_1(t)})g(\frac{x}{\lambda_2(t)})]|_{E_{0,t}}\\
=&\frac{1}{\eta(t)}\int_{E_{+,t}}\partial_xv_n\chi_{E_{+,t}}(x)\partial^2_x[w(\frac{x}{\lambda_1(t)})g(\frac{x}{\lambda_2(t)})]dx.
\end{split}
\end{equation*}


We claim that
\begin{equation}\label{s-6}
|I\!I|<C(\|u_{n0}\|_{H^1(\mathbb{R})},\|v_{n0}\|_{H^1(\mathbb{R})}).
\end{equation}

Indeed, using the Cauchy-Schwarz inequality and the restrictions $r_1+p_1=1, \;r_1>0$,
\begin{equation*}
\lambda_1(t)=t^{p_1}<\lambda_2(t)=t^{p_1p_2} \quad \text{ with}\quad p_2>2,
\end{equation*}
we have
\begin{equation*}
\begin{split}
|I\!I|=&|\int^{+\infty}_{2}\frac{1}{\eta(t)}\int_{E_{+,t}}\partial_xv_n\chi_{E_{+,t}}(x)\partial^2_x[w(\frac{x}{\lambda_1(t)})g(\frac{x}{\lambda_2(t)})]dxdt|\\
\lesssim& |\int^{+\infty}_{2}\frac{1}{\eta(t)\lambda^2_1(t)}\int_{\mathbb{R}}|\partial_xv_n|w''(\frac{x}{\lambda_1(t)})g(\frac{x}{\lambda_2(t)})dxdt|\\
&+|\int^{+\infty}_{2}\frac{1}{\eta(t)\lambda^2_2(t)}\int_{\mathbb{R}}|\partial_xv_n|w(\frac{x}{\lambda_1(t)})g''(\frac{x}{\lambda_2(t)})dxdt|\\
&+|\int^{+\infty}_{2}\frac{1}{\eta(t)\lambda_2(t)\lambda_1(t)}\int_{\mathbb{R}}|\partial_xv_n|w'(\frac{x}{\lambda_1(t)})g'(\frac{x}{\lambda_2(t)})dxdt|
\end{split}
\end{equation*}
\begin{equation*}
\begin{split}
\lesssim &\int^{+\infty}_{2}\frac{\lambda^{\frac{1}{2}}_1(t)}{\eta(t)\lambda^2_1(t)}\|\partial_xv_n\|_{L^2_x}(t)\|w''\|_{L^2_x}dt+ \int^{+\infty}_{2}\frac{\lambda^{\frac{1}{2}}_2(t)}{\eta(t)\lambda^2_2(t)}\|\partial_xv_n\|_{L^2_x}(t)\|g''\|_{L^2_x}dt\\
&+ \int^{+\infty}_{2}\frac{\lambda^{\frac{1}{2}}_2(t)}{\eta(t)\lambda_1(t)\lambda_2(t)}\|\partial_xv_n\|_{L^2_x}(t)\|g'\|_{L^2_x}dt\\
\lesssim& \sup_{t\in (2,+\infty)}\|\partial_xv_n\|_{L^2_x}(t) \int^{+\infty}_{2}\Big( \frac{\lambda^{\frac{1}{2}}_1(t)}{\eta(t)\lambda^2_1(t)}+\frac{\lambda^{\frac{1}{2}}_1(t)}{\eta(t)\lambda^2_1(t)}+ \frac{\lambda^{\frac{1}{2}}_2(t)}{\eta(t)\lambda_1(t)\lambda_2(t)} \Big)dt\\
\lesssim& C(\|u_{n0}\|_{H^1(\mathbb{R})}, \|v_{n0}\|_{H^1(\mathbb{R})})\int^{+\infty}_{2}\frac{1}{t^{r_1+2p_1}}dt\\
\lesssim& C(\|u_{n0}\|_{H^1(\mathbb{R})}, \|v_{n0}\|_{H^1(\mathbb{R})}).
\end{split}
\end{equation*}

\medskip

For $I$: We first use the Fubini theorem to change the order of integration, i.e.,
\begin{equation}\label{s-7}
\begin{split}
I=&\int^{+\infty}_2\int_{\mathbb{R}}\partial_tv_n\chi_{E_{+,t}}(x)\frac{1}{\eta(t)}w(\frac{x}{\lambda_1(t)})g(\frac{x}{\lambda_2(t)})dxdt\nonumber\\
=&\int_{\mathbb{R}}\int^{+\infty}_2\partial_tv_n\chi_{E_{+,x}}(t)\frac{1}{\eta(t)}w(\frac{x}{\lambda_1(t)})g(\frac{x}{\lambda_2(t)})dtdx.
\end{split}
\end{equation}
We note that we have changed the factor $\chi_{E_{+,t}}(x)$ as $\chi_{E_{+,x}}(t)$ with \[E_{+,x}=\{t:\frac{v^2_n}{2}(x,t)-\gamma|u_n|^2(x,t)>0\},\] in the expression of $I$ since we have changed the integration orders.
Using integrating by parts with respect to the time, it follows
\begin{equation*}
\begin{split}
I=&\int_{\mathbb{R}}\int^{+\infty}_2\partial_tv_n\chi_{E_{+,x}}(t)\frac{1}{\eta(t)}w(\frac{x}{\lambda_1(t)})g(\frac{x}{\lambda_2(t)})dtdx\\
=&-\int_{\mathbb{R}}\int_{(2,+\infty)\cap E_{+,x}}v_n\partial_t[\chi_{E_{+,x}}(t)\frac{1}{\eta(t)}w(\frac{x}{\lambda_1(t)})g(\frac{x}{\lambda_2(t)})]dtdx\\
&+\underbrace{\int_{\mathbb{R}}\left[v_n\chi_{E_{+,x}}(t)\frac{1}{\eta(t)}w(\frac{x}{\lambda_1(t)})g(\frac{x}{\lambda_2(t)})\right]|_{t\in \partial(E_{+,x}\cap [2,+\infty))}dx}_{I_a},
\end{split}
\end{equation*}
where $\partial(E_{+,x}\cap [2,+\infty))$ denotes the boundary of $E_{+,x}\cap [2,+\infty)$.

Since
\begin{equation*}
\begin{split}
\partial(E_{+,x}\cap [2,+\infty))&\subset E_{0,x}\cup \{2\}\cup\{+\infty\},\\
\chi_{E_{+,x}}(t)|_{E_{0,x}}&\equiv 0,\\
\lim_{t\rightarrow +\infty}\frac{1}{\eta(t)}w(\frac{x}{\lambda_1(t)})g(\frac{x}{\lambda_2(t)})&=\frac{\pi}{8}\lim_{t\rightarrow +\infty}\frac{1}{\eta(t)}=0,
\end{split}
\end{equation*}
applying the Cauchy-Schwarz inequality and Lemma \ref{result6} it follows that
\begin{equation}\label{s-21}
\begin{split}
|I_a|=&|\int_{\mathbb{R}}\left[v_n\chi_{E_{+,x}}(t)\frac{1}{\eta(t)}w(\frac{x}{\lambda_1(t)})g(\frac{x}{\lambda_2(t)})\right]|_{t\in \partial(E_{+,x}\cap [2,+\infty))}dx|\\
&\lesssim  \int_{\mathbb{R}}|v_n(x,2)|g(\frac{x}{\lambda_1(2)})dx\\
& \lesssim \|v_n(\cdot,2)\|_{L^2(\mathbb{R})}\|g\|_{L^2(\mathbb{R})}(\lambda_1(2))^{\frac{1}{2}}\\
&\lesssim C(\|u_{n0}\|_{H^1(\mathbb{R})},\|v_{n0}\|_{H^1(\mathbb{R})} ).
\end{split}
\end{equation}
Since $\partial_t\chi_{E_{+,x}}(t)=0$ for any $t\in (2,+\infty)\cap E_{+,x}$ and the fact $E_{+,x}$ is open\footnote{We can prove that $u_n(x,t)$, $v_n(x,t)\in C_x(\mathbb{R}; C_t(\mathbb{R}^{+}))$ since $u_n(x,t)$, $v_n(x,t)\in C_t(\mathbb{R}^{+}; H^1_x(\mathbb{R}))$. Then $\frac{v^2_n}{2}(x,t)-\gamma|u_n|^2(x,t)$ is continuous on the time for any fixed $x\in \mathbb{R}$. Hence, the set $E_{+,x}$ is open by the local preserving sign properties of continuous functions.}, we now decompose $I$ as
\begin{equation}\label{s-8}
\begin{split}
I=&-\int_{\mathbb{R}}\int_{(2,+\infty)\cap E_{+,x}}v_n\partial_t\cancel{\chi_{E_{+,x}}(t)}\frac{1}{\eta(t)}w(\frac{x}{\lambda_1(t)})g(\frac{x}{\lambda_2(t)})dtdx \\
&-\int_{\mathbb{R}}\int_{(2,+\infty)\cap E_{+,x}}v_n\chi_{E_{+,x}}(t)\partial_t[\frac{1}{\eta(t)}w(\frac{x}{\lambda_1(t)})g(\frac{x}{\lambda_2(t)})]dtdx \\
&+I_a\\
=&-\int_{\mathbb{R}}\int_{(2,+\infty)\cap E_{+,x}}v_n\chi_{E_{+,x}}(t)\partial_t[\frac{1}{\eta(t)}w(\frac{x}{\lambda_1(t)})g(\frac{x}{\lambda_2(t)})]dtdx+I_a.
\end{split}
\end{equation}

\medskip

Now employing that
\begin{equation*}
\begin{split}
\partial_t[\frac{1}{\eta(t)}w(\frac{x}{\lambda_1(t)})g(\frac{x}{\lambda_2(t)})]
=&-\frac{\eta'(t)}{\eta(t)}\frac{1}{\eta(t)}w(\frac{x}{\lambda_1(t)})g(\frac{x}{\lambda_2(t)})-\frac{\lambda_1'(t)}{\lambda_1(t)}\frac{1}{\eta(t)}w'(\frac{x}{\lambda_1(t)})\frac{x}{\lambda_1(t)}g(\frac{x}{\lambda_2(t)})\\
&-\frac{\lambda_2'(t)}{\lambda_2(t)}\frac{1}{\eta(t)}w(\frac{x}{\lambda_1(t)})g'(\frac{x}{\lambda_2(t)})\frac{x}{\lambda_2(t)},
\end{split}
\end{equation*}
using the Fubini theorem, the Cauchy-Schwarz inequality in the last term in \eqref{s-8}, and the facts
\[\lambda_1(t)=t^{p_1},\quad \lambda_2(t)=t^{p_1p_2},\quad \eta(t)=t^{r_1},\]
\[\frac{\eta'(t)}{\eta(t)}\backsim \frac{\lambda_1'(t)}{\lambda_1(t)}\backsim \frac{\lambda_2'(t)}{\lambda_2(t)}\backsim\frac{1}{t},\]
\[w(\cdot)\in L^{\infty}(\mathbb{R}),w'(\cdot)(\cdot),g(\cdot),g'(\cdot)(\cdot)\in L^2(\mathbb{R}),\]
\begin{equation*}
\sup_{t\in(2,+\infty),x\in \mathbb{R}}\|v_n\|\lesssim \|v_n\|_{C(2,+\infty;H^1(\mathbb{R}))}\lesssim C(\|v_{n0}\|_{H^1(\mathbb{R})},\|u_{n0}\|_{H^1(\mathbb{R})}),
\end{equation*}
we deduce that
\begin{equation}\label{s-9}
\begin{split}
|I|=&|\int_{\mathbb{R}}\int_{(2,+\infty)\cap E_{+,x}}v_n\chi_{E_{+,x}}(t)\partial_t[\frac{1}{\eta(t)}w(\frac{x}{\lambda_1(t)})g(\frac{x}{\lambda_2(t)})]dtdx+I_a|\\
\leq& \int_{\mathbb{R}}\int_{(2,+\infty)\cap E_{+,x}}|v_n\chi_{E_{+,x}}(t)\partial_t[\frac{1}{\eta(t)}w(\frac{x}{\lambda_1(t)})g(\frac{x}{\lambda_2(t)})]|dtdx+|I_a|\\
\leq& \int^{+\infty}_{2}\int_{\mathbb{R}}|v_n\partial_t[\frac{1}{\eta(t)}w(\frac{x}{\lambda_1(t)})g(\frac{x}{\lambda_2(t)})]|dxdt+|I_a|\\
\lesssim & \; C(\|v_{n0}\|_{H^1(\mathbb{R})},\|u_{n0}\|_{H^1(\mathbb{R})})\int^{\infty}_2\frac{1}{t}\frac{1}{t^{r_1-\frac{p_1p_2}{2}}}dt+|I_a|.
\end{split}
\end{equation}

Now, the conditions $p_1+r_1=1,\; p_1>0$, and
\begin{equation}
0<p_1<\frac{2}{p_2+2},\hskip20pt p_2>1,
\end{equation}
guarantee the finiteness of the last integral on the right hand side of  \eqref{s-9}. Combining this and \eqref{s-21} yield
\begin{equation}\label{s-10}
|I|\lesssim \;C(\|v_{n0}\|_{H^1(\mathbb{R})},\|u_{n0}\|_{H^1(\mathbb{R})}),\quad\quad \forall n\geq 1.
\end{equation}

\medskip

Gathering the information obtained in \eqref{s-4}, and substituting \eqref{s-5}, \eqref{s-6} and \eqref{s-10} into \eqref{s-2}, we have that
\begin{equation}\label{step1res}
\begin{split}
\int^{\infty}_2\frac{1}{\eta(t)\lambda_1(t)}&\int_{E_{+,t} }(\frac{1}{2}v^2_n-\gamma |u_n|^2 )\chi_{E_{+,t}}(x)w'(\frac{x}{\lambda_1(t)})g(\frac{x}{\lambda_2(t)})dxdt\\
=&-I-I\!I-I\!I\!I_{a}\\
\lesssim & \; C(\|u_{n0}\|_{H^1(\mathbb{R})},\|v_{n0}\|_{H^1(\mathbb{R})}).
\end{split}
\end{equation}

\medskip

A similar argument leads to
\begin{equation*}
\begin{split}
\int^{\infty}_2\frac{1}{\eta(t)\lambda_1(t)}&\int_{E_{-,t} }(\frac{1}{2}v^2_n-\gamma |u_n|^2 )\chi_{E_{-,t}}(x)w'(\frac{x}{\lambda_1(t)})g(\frac{x}{\lambda_2(t)})dxdt\\
\lesssim & \,C(\|u_{n0}\|_{H^1(\mathbb{R})},\|v_{n0}\|_{H^1(\mathbb{R})}).
\end{split}
\end{equation*}

Thus, combining  the above two estimates, using the definitions of $E_{-,t}$, $E_{0,t}$, $E_{+,t}$, and the identity $\eta(t)\lambda_1(t)=t^{p_1+r_1}=t$, we arrive at
\begin{equation}\label{1step1res}
\int^{\infty}_2\frac{1}{t}\int_{\mathbb{R}}|\frac{1}{2}v^2_n-\gamma |u_n|^2 |w'(\frac{x}{\lambda_1(t)})g(\frac{x}{\lambda_2(t)})dxdt\lesssim  C(\|u_{n0}\|_{H^1(\mathbb{R})},\|v_{n0}\|_{H^1(\mathbb{R})}).
\end{equation}

\vspace{3mm}

\textbf{\textbf{Step 2:  Construct a regularized sequence of solutions to the Schr\"odinger-KdV system \eqref{Equ(0.1)} with the initial data $(u_{0},v_{0})\in (H^1(\mathbb{R}))^2$}.}

\medskip

Let $\zeta\in C_0^{\infty}(\mathbb R)$, $\zeta(x)\equiv 1$, if $|x|\le 1$; $\zeta(x)=0$, if $|x|\ge 2 $
satisfying $\int_{\mathbb{R}}\zeta(x)dx=1$.
We then introduce the mollifier sequences $\{\zeta_{n}(x)\}^{+\infty}_{n=1}$ by letting
\begin{equation*}
\zeta_{n}(x)=n\zeta(n x),\quad\quad\quad n=1,2,3,\dots
\end{equation*}


For given initial data $(u_{0},v_{0})\in H^1(\mathbb{R})\times H^1(\mathbb{R})$ in \eqref{Equ(0.1)}, we construct an approximate initial data sequence through
\begin{equation*}
u_{n0}(x)=u_0(x)\ast \zeta_{n}(x),\quad\quad v_{n0}(x)=v_0(x)\ast \zeta_{n}(x).
\end{equation*}
It can be readily checked that
\begin{equation}\label{s-11}
\lim_{n\rightarrow +\infty}\|u_{n0}-u_{0}\|_{H^1(\mathbb{R})}+\lim_{n\rightarrow +\infty}\|v_{n0}-v_{0}\|_{H^1(\mathbb{R})}=0,
\end{equation}
and
\begin{equation}\label{s-12}
(u_{n0},v_{n0})\in H^{s}(\mathbb{R})\times H^{s}(\mathbb{R}),\quad \forall s>4,\quad \forall n\in \mathbb{Z}^{+}.\end{equation}
Hence the  solution $(u_{n},v_{n})$ to the Schr\"odinger-KdV system \eqref{Equ(0.1)} with the initial data
$(u_{n0},v_{n0})$ is a classical global solution \cite{MSA1993Tsutsumi}.

Let $(u,v)$ be the weak solution of the Schr\"odinger-KdV system \eqref{Equ(0.1)} with the initial data
$(u_{0},v_{0})\in (H^1(\mathbb{R}))^2$. We claim that:
\begin{description}
  \item[Claim 1] \begin{equation}\label{s-13}
  (u_{n},v_{n})\rightarrow (u,v) \text{ in } (C(0,+\infty; H^1(\mathbb{R})))^2 \text{ as } n\rightarrow +\infty.
  \end{equation}
\item[Claim 2] \begin{equation}\label{s-14}
  \lim_{n\rightarrow+\infty}\int^{\infty}_{2}\int_{\mathbb{R}}|v^2_n-v^2|\frac{w'(\frac{x}{\lambda_1(t)})g(\frac{x}{\lambda_2(t)})}{t}dxdt=0.
  \end{equation}
\item[Claim 3] \begin{equation}\label{s-15}
    \lim_{n\rightarrow+\infty}\int^{\infty}_{2}\int_{\mathbb{R}}||u_n|^2-|u|^2|\frac{w'(\frac{x}{\lambda_1(t)})g(\frac{x}{\lambda_2(t)})}{t}dxdt=0.
    \end{equation}
\end{description}

We prove Claims 1 and 2 by using a contradiction argument. Since the argument to show Claim 3 is similar to the one employed to establish Claim 2, we omit the details of the proof of Claim 3.\\

\underline{\it Proof of Claim 1:} Indeed,  Corcho-Linares's local Lipschitz continuous argument of the solution map
$(u(x,0),v(x,0))\mapsto (u(x,t),v(x,t))$ (see Section 4 in \cite{TAMS2007Linares}) and \eqref{s-11} imply  that, for any fixed small $\epsilon>0$, there exist $N=N(\epsilon)>0$, for all $n\geq N$ and an appropriate local existence time interval $[0,T]$, satisfying
\begin{equation}\label{1s-16}
\begin{split}
\|u_{n}-u\|_{C([0,T];H^1(\mathbb{R}))}+\|v_{n}-v\|_{C([0,T];H^1(\mathbb{R}))}\leq & c_0\left(\|u_{n0}-u_{0}\|_{H^1(\mathbb{R})}+\|v_{n0}-v_{0}\|_{H^1(\mathbb{R})}\right)\\
\leq & c_0\epsilon.
\end{split}
\end{equation}
Using Lemma \ref{result6} and making a small modification of the local well-posedness argument in \cite{TAMS2007Linares}, we note that the local time $T$ in \eqref{1s-16} is actually independent of $\epsilon$ and $n$.

For any $k\in \mathbb{Z}^{+}$, we can use the previous argument to deduce that
\begin{equation}\label{s-16}
\begin{split}
\|u_{n}-&u\|_{C([kT,(k+1)T];H^1(\mathbb{R}))}+\|v_{n}-v\|_{C([kT,(k+1)T];H^1(\mathbb{R}))}\\
\leq &\; c_0\|u_{n}-u\|_{H^1(\mathbb{R})}(kT)+c_0\|v_{n}-v\|_{H^1(\mathbb{R})}(kT)\\
\leq &\; c^2_0\left(\|u_{n}-u\|_{H^1(\mathbb{R})}((k-1)T)+\|v_{n}-v\|_{H^1(\mathbb{R})}((k-1)T) \right)\\
\leq &\;c^{k+1}_0\left(\|u_{n0}-u_0\|_{H^1(\mathbb{R})}+\|v_{n0}-v_0\|_{H^1(\mathbb{R})} \right)\\
\leq &\; c^{k+1}_0\epsilon.
\end{split}
\end{equation}

Suppose that
$(u_{n},v_{n})$ does not approximate to $(u,v)$ in  $C(0,+\infty; H^1(\mathbb{R}))$  as $n\rightarrow +\infty.$ On one hand, there exist $\epsilon_0>0$, $k_0\in \mathbb{Z}^{+}\cup \{0\}$ and $t_0\in [k_0T,(k_0+1)T]$, for any $N\in \mathbb{Z}^{+}$ and there exists $\widetilde{n}_N\in \mathbb{Z}^{+}$ satisfying
\begin{equation*}
\epsilon_0<\|u_{\widetilde{n}_N}-u\|_{C([k_0T,(k_0+1)T];H^1(\mathbb{R}))}+\|v_{\widetilde{n}_N}-v\|_{C([k_0T,(k_0+1)T];H^1(\mathbb{R}))}.
\end{equation*}
On the other hand, using \eqref{s-16}, one has
\begin{equation*}
\|u_{\widetilde{n}_N}-u\|_{C([k_0T,(k_0+1)T];H^1(\mathbb{R}))}+\|v_{\widetilde{n}_N}-v\|_{C([k_0T,(k_0+1)T];H^1(\mathbb{R}))}\leq  c^{k_0+1}_0\epsilon.
\end{equation*}
Therefore, the parameter $\epsilon>\frac{\epsilon_0}{c^{k_0+1}_0}$. But this contradicts the smallness assumption on
$\epsilon$. Hence, we establish Claim 1.\\

\underline{\it Proof of Claim 2:} Suppose that Claim 2 does not hold. There exists a constant $\epsilon_0>0$, for any $N\in \mathbb{Z}^{+}$, such that we can find an $n>N$ satisfying
\begin{equation*}
\begin{split}
\sum^{+\infty}_{k=2}\int^{(k+1) }_{k}\int_{\mathbb{R}}|v^2_n-v^2|\frac{w'(\frac{x}{\lambda_1(t)})g(\frac{x}{\lambda_2(t)})}{t}dxdt
=&\int^{\infty}_{2}\int_{\mathbb{R}}|v^2_n-v^2|\frac{w'(\frac{x}{\lambda_1(t)})g(\frac{x}{\lambda_2(t)})}{t}dxdt\\
>&\;\epsilon_0.
\end{split}
\end{equation*}
Due to the fact that $\frac{w'(\frac{x}{\lambda_1(t)})g(\frac{x}{\lambda_2(t)})}{t}$ is nonnegative on $(2,+\infty)\times \mathbb{R}$ and $\sum^{+\infty}_{k=1}\frac{1}{k^2}=\frac{\pi^2}{6}$, there must exist $k_0\in \mathbb{N}$ such that

 \begin{equation}\label{s-17}
 k^2_0\int^{k_0+1}_{k_0}\int_{\mathbb{R}}|v^2_n-v^2|\frac{w'(\frac{x}{\lambda_1(t)})g(\frac{x}{\lambda_2(t)})}{t}dxdt>\frac{6}{\pi^2}\epsilon_0.
\end{equation}
Recalling that $v_n,v\in C(0,+\infty;H^1(\mathbb{R}))$, we have
\begin{equation}\label{s-18}
\begin{split}
\int^{k_0+1}_{k_0}\int_{\mathbb{R}}&|v^2_n-v^2|\frac{w'(\frac{x}{\lambda_1(t)})g(\frac{x}{\lambda_2(t)})}{t}dxdt\\
\leq & \int^{k_0+1}_{k_0}\frac{1}{t}\|v^2_n-v^2\|_{L^1(\mathbb{R})}dt\\
\leq &\sup_{t\in (0,+\infty)}\|v_n+v\|_{L^2(\mathbb{R})}(t)\sup_{t\in (0,+\infty)}\|v_n-v\|_{L^2(\mathbb{R})}(t)\ln(1+\frac{1}{k_0})\\
\leq &\;C(\|u_{n0}+u_{0}\|_{H^1(\mathbb{R})},\|v_{n0}+v_{0}\|_{H^1(\mathbb{R})})\sup_{t\in (0,+\infty)}\|v_n-v\|_{L^2(\mathbb{R})}(t)\ln(1+\frac{1}{k_0}),
\end{split}
\end{equation}
where the constant $C(\|u_{n0}+u_{0}\|_{H^1(\mathbb{R})},\|v_{n0}+v_{0}\|_{H^1(\mathbb{R})})$ is bounded
from the definitions of $u_{n0}$, $v_{n0}$ and the argument given in section 4 of \cite{TAMS2007Linares}. More precisely,
$$
C(\|u_{n0}+u_{0}\|_{H^1(\mathbb{R})},\|v_{n0}+v_{0}\|_{H^1(\mathbb{R})})\leqslant C(\|u_0\|_{H^1(\mathbb{R})},\|v_0\|_{H^1(\mathbb{R})})
$$
if $n$ is sufficiently large.

Combining \eqref{s-17} and \eqref{s-18}, we have
\begin{align}\sup_{t\in (0,+\infty)}\|v_n-v\|_{L^2(\mathbb{R})}(t)&\geq\frac{6\epsilon_0}{\pi^2k^2_0C(\|u_0\|_{H^1(\mathbb{R})},\|v_0\|_{H^1(\mathbb{R})})\ln(1+\frac{1}{k_0})},
\end{align}
which contradicts Claim 1.

\medskip

\textbf{Step 3:  Construct $\int^{\infty}_{2}\frac{1}{t}\int_{\mathbb{R}}|\frac{v^2}{2}-\gamma|u|^2|w'(\frac{x}{\lambda_1(t)})g(\frac{x}{\lambda_2(t)})dxdt$ through the regularized sequence solution established in Step 2.}

\medskip

For the regularized sequence solution $\{(u_n,v_n)\}^{+\infty}_{n=1}$ constructed in step 2 for the solution $(u,v)$ of the Schr\"odinger-KdV system \eqref{Equ(0.1)} with the initial data $(u_0,v_0)\in (H^1(\mathbb{R}))^2$, we have
\begin{equation}\label{s-19}
\begin{split}
\int^{+\infty}_2\frac{1}{t}\int_{\mathbb{R}}&|\frac{v^2}{2}-\gamma|u|^2|w'(\frac{x}{\lambda_1(t)})g(\frac{x}{\lambda_2(t)})dxdt\\
=&\int^{+\infty}_2\frac{1}{t}\int_{\mathbb{R}}|\frac{v^2}{2}-\frac{v^2_n}{2}+\gamma |u_n|^2 -\gamma|u|^2+\frac{v^2_n}{2}-\gamma |u_n|^2|w'(\frac{x}{\lambda_1(t)})g(\frac{x}{\lambda_2(t)})dxdt\\
\leq&\int^{+\infty}_2\frac{1}{t}\int_{\mathbb{R}}|\frac{v^2}{2}-\frac{v^2_n}{2}|w'(\frac{x}{\lambda_1(t)})g(\frac{x}{\lambda_2(t)})dxdt\\
&+\int^{+\infty}_2\frac{1}{t}\int_{\mathbb{R}}|\gamma |u_n|^2 -\gamma|u|^2|w'(\frac{x}{\lambda_1(t)})g(\frac{x}{\lambda_2(t)})dxdt\\
&+\int^{+\infty}_2\frac{1}{t}\int_{\mathbb{R}}|\frac{v^2_n}{2}-\gamma |u_n|^2|w'(\frac{x}{\lambda_1(t)})g(\frac{x}{\lambda_2(t)})dxdt.
\end{split}
\end{equation}

Using Claims 2 and 3, for any $\epsilon>0$, there exists $N=N(\epsilon)$ such that for $n\geq N$ we
have
\begin{equation*}
\int^{+\infty}_2\frac{1}{t}\int_{\mathbb{R}}|\frac{v^2}{2}-\frac{v^2_n}{2}|w'(\frac{x}{\lambda_1(t)})g(\frac{x}{\lambda_2(t)})dxdt<\epsilon,
\end{equation*}
and
\begin{equation*}
\int^{+\infty}_2\frac{1}{t}\int_{\mathbb{R}}|\gamma |u_n|^2 -\gamma|u|^2|w'(\frac{x}{\lambda_1(t)})g(\frac{x}{\lambda_2(t)})dxdt<\epsilon.
\end{equation*}

Recalling \eqref{1step1res} and \eqref{s-11},
\begin{equation*}
\begin{split}
\int^{+\infty}_2\frac{1}{t}&\int_{\mathbb{R}}|\frac{1}{2}v^2_n-\gamma |u_n|^2 |w'(\frac{x}{\lambda_1(t)})g(\frac{x}{\lambda_2(t)})dxdt\\
\lesssim &\; C(\|u_{n0}\|_{H^1(\mathbb{R})},\|v_{n0}\|_{H^1(\mathbb{R})})\\
\lesssim &\; C(\|u_{0}\|_{H^1(\mathbb{R})},\|v_{0}\|_{H^1(\mathbb{R})}).
\end{split}
\end{equation*}

The remarks above combined with \eqref{s-19} give us that
\begin{equation}\label{s-20}
\int^{+\infty}_2\frac{1}{t}\int_{\mathbb{R}}|\frac{v^2}{2}-\gamma|u|^2|w'(\frac{x}{\lambda_1(t)})g(\frac{x}{\lambda_2(t)})dxdt<C(\|u_{0}\|_{H^1(\mathbb{R})},\|v_{0}\|_{H^1(\mathbb{R})}).
\end{equation}

$\hfill{} \Box$\\

Using the same idea as in the proof of Proposition \ref{prop1}, we have the following result.

\begin{proposition}\label{2prop2} Under the same conditions on the weight functions $w(x)$ and $g(x)$, and the indices $r_1$,$p_1$,$p_2$ in Proposition \ref{prop1}, we have
\begin{equation}\label{2result}
\int^{+\infty}_2\frac{1}{t} \int_{\mathbb {R}}|\Re u(\alpha v+\beta |u|^2)| w'\left(\frac{x}{\lambda_{1}(t)}\right)g\left(\frac{x}{\lambda_{2}(t)}\right) d xdt\leq C(\|u_0\|_{H^1(\mathbb{R})},\|v_0\|_{H^1(\mathbb{R})}),
\end{equation}
and
\begin{equation}\label{bb1}
\int^{+\infty}_2\frac{1}{t} \int_{\mathbb {R}}|\Im u(\alpha v+\beta |u|^2)| w'\left(\frac{x}{\lambda_{1}(t)}\right)g\left(\frac{x}{\lambda_{2}(t)}\right) d xdt\leq C(\|u_0\|_{H^1(\mathbb{R})},\|v_0\|_{H^1(\mathbb{R})})
\end{equation}
where $\Re u$ and $\Im u$  denote the real and imaginary parts of the complex valued function $u$, respectively.
 \end{proposition}

\begin{proof} We first construct the regularized sequence solution $\{(u_n,v_n)\}^{+\infty}_{n=1}$ to the Schr\"odinger-KdV system \eqref{Equ(0.1)} with the initial data $(u_{n0},v_{n0})\in H^s(\mathbb{R})\times H^s(\mathbb{R})$ where $s>4$  satisfying
\begin{equation}\label{c3}\lim_{n\rightarrow +\infty}\|u_{n0}-u_{0}\|_{H^1(\mathbb{R})}+\lim_{n\rightarrow +\infty}\|v_{n0}-v_{0}\|_{H^1(\mathbb{R})}=0.\end{equation}

From the first equation in \eqref{smooth-sol}, i.e.
$i\partial_tu_n+\partial^2_xu_n=u_n(\alpha v_n+\beta|u_n|^2)$, we have
\begin{equation}\label{baR}
-\partial_t \Im u_n+\partial^2_x\Re u_n=\Re u_n(\alpha v_n+\beta|u_n|^2).
\end{equation}

Next, we introduce
\[\widetilde{E}_{-,t}=\{x:\Re u_n(\alpha v_n+\beta|u_n|^2)(x,t)<0\},\]
\[\widetilde{E}_{0,t}=\{x:\Re u_n(\alpha v_n+\beta|u_n|^2)(x,t)=0\},\]
\[\widetilde{E}_{+,t}=\{x:\Re u_n(\alpha v_n+\beta|u_n|^2)(x,t)>0\},\]
then applying the techniques to get \eqref{1step1res}, we deduce that \footnote{One can utilize the same techniques in \eqref{s-21}-\eqref{1step1res}.}
\begin{equation*}
\begin{split}
\int^{\infty}_2\frac{1}{\eta(t)\lambda_1(t)}&\int_{\widetilde{E}_{-,t} }(\Re u_n(\alpha v_n+\beta|u_n|^2))
\chi_{\widetilde{E}_{-,t}}(x)w'(\frac{x}{\lambda_1(t)})g(\frac{x}{\lambda_2(t)})dxdt\\
\lesssim &\; C(\|u_{n0}\|_{H^1(\mathbb{R})},\|v_{n0}\|_{H^1(\mathbb{R})}),
\end{split}
\end{equation*}
and
\begin{equation*}
\begin{split}
\int^{\infty}_2\frac{1}{\eta(t)\lambda_1(t)}&\int_{\widetilde{E}_{+,t} }(\Re u_n(\alpha v_n+\beta|u_n|^2))\chi_{\widetilde{E}_{+,t}}(x)w'(\frac{x}{\lambda_1(t)})g(\frac{x}{\lambda_2(t)})dxdt\\
\lesssim & \;C(\|u_{n0}\|_{H^1(\mathbb{R})},\|v_{n0}\|_{H^1(\mathbb{R})}).
\end{split}
\end{equation*}
Recalling $\eta(t)\lambda_1(t)=t$, the above two estimates imply that
\begin{equation}\label{ba1}
\begin{split}
\int^{\infty}_2\frac{1}{t}&\int_{\mathbb{R} }|\Re u_n(\alpha v_n+\beta|u_n|^2)|w'(\frac{x}{\lambda_1(t)})g(\frac{x}{\lambda_2(t)})dxdt\\
\lesssim&\; C(\|u_{n0}\|_{H^1(\mathbb{R})},\|v_{n0}\|_{H^1(\mathbb{R})})\\
\lesssim &\; C(\|u_{0}\|_{H^1(\mathbb{R})},\|v_{0}\|_{H^1(\mathbb{R})}).
\end{split}
\end{equation}

We write
\begin{equation}\label{c2}
\begin{split}
\Re u(\alpha v+\beta |u|^2)=&\;\Re u_n(\alpha v+\beta |u|^2)+\Re (u-u_n)(\alpha v+\beta |u|^2)\\
=&\;\Re u_n(\alpha v_n+\beta |u_n|^2)+\Re u_n(\alpha (v-v_n)-\beta |u_n|^2+\beta|u|^2)\\
&\;+\Re (u-u_n)(\alpha v+\beta |u|^2).
\end{split}
\end{equation}
Then by combining the two {\it a priori} estimates
\begin{equation*}
\begin{split}
\sup _{t\in[0,+\infty)}\left(\|u(t)\|_{H^{1}}+\|v(t)\|_{H^{1}}\right)
&\leq C \left(\left\|u_{0}\right\|_{H^{1}},\left\|v_{0}\right\|_{H^{1}}\right),\\
\sup _{t\in[0,+\infty)}\left(\|u_n(t)\|_{H^{1}}+\|v_n(t)\|_{H^{1}}\right) &\leq C \left(\left\|u_{n0}\right\|_{H^{1}},\left\|v_{n0}\right\|_{H^{1}}\right)\lesssim C \left(\left\|u_{0}\right\|_{H^{1}},\left\|v_{0}\right\|_{H^{1}}\right),
\end{split}
\end{equation*}
and the embedding $H^1(\mathbb{R})\hookrightarrow C(\mathbb{R})$, we easily get from \eqref{c2} that
\begin{equation}\label{c1}
\begin{split}
&|\Re u(\alpha v+\beta |u|^2)|\\
\lesssim &\;|\Re u_n(\alpha v_n+\beta |u_n|^2)|+|u_n(\alpha (v-v_n)-\beta |u_n|^2+\beta|u|^2 )|+|(u-u_n)(\alpha v+\beta |u|^2)|\\
\lesssim &\;|\Re u_n(\alpha v_n+\beta |u_n|^2)|+\|u_n\|_{L^{\infty}_x}(|v-v_n|+(|u_n|+|u|)(||u_n|-|u||)|\\
&\;+(\|v\|_{L^{\infty}_x}+\|u\|^2_{L^{\infty}_x})|u-u_n|\\
\lesssim &\;|\Re u_n(\alpha v_n+\beta |u_n|^2)|+C(\|u_0\|_{H^1},\|v_0\|_{H^1})\left(|v-v_n|+|u-u_n|\right)\\
\lesssim &\;|\Re u_n(\alpha v_n+\beta |u_n|^2)|+\left(|v-v_n|+|u-u_n|\right).
\end{split}
\end{equation}

Claim 1, i.e.,
\begin{equation*}
(u_{n},v_{n})\rightarrow (u,v) \text{ in } (C(0,+\infty; H^1(\mathbb{R})))^2 \text{ as } n\rightarrow +\infty,
\end{equation*}
and the argument used to prove Claim 2 allow us to deduce that
\begin{equation}\label{ss-14}
\lim_{n\rightarrow+\infty}\int^{\infty}_{2}\int_{\mathbb{R}}\left(|v-v_n|+|u-u_n|\right)\frac{w'(\frac{x}{\lambda_1(t)})g(\frac{x}{\lambda_2(t)})}{t}dxdt=0.
\end{equation}

Thus, gathering together \eqref{c1}, \eqref{ba1}, \eqref{ss-14}, and \eqref{c3}, we obtain
\begin{equation}\label{ss-19}
\begin{split}
\int^{+\infty}_2\frac{1}{t}&\int_{\mathbb{R}}|\Re u(\alpha v+\beta |u|^2)|w'(\frac{x}{\lambda_1(t)})g(\frac{x}{\lambda_2(t)})dxdt\\
\leq&\;\int^{+\infty}_2\frac{1}{t}\int_{\mathbb{R}}|\Re u_n(\alpha v_n+\beta |u_n|^2)|w'(\frac{x}{\lambda_1(t)})g(\frac{x}{\lambda_2(t)})dxdt\\
&\;+\int^{+\infty}_2\frac{1}{t}\int_{\mathbb{R}}\left(|v-v_n|+|u-u_n|\right)w'(\frac{x}{\lambda_1(t)})g(\frac{x}{\lambda_2(t)})dxdt\\
\lesssim&\; C(\|u_{n0}\|_{H^1(\mathbb{R})},\|v_{n0}\|_{H^1(\mathbb{R})})+1\\
\lesssim &\; C(\|u_{0}\|_{H^1(\mathbb{R})},\|v_{0}\|_{H^1(\mathbb{R})}),
\end{split}
\end{equation}
which implies \eqref{2result}.

The proof of \eqref{bb1} follows the same lines as that of \eqref{2result}.
The main difference is that we do not consider  \eqref{baR} but
\begin{equation}\label{baI}
\partial_t\Re u_n+\partial^2_x\Im u_n=\Im u_n(\alpha v_n+\beta|u_n|^2),
\end{equation}
which can be obtained by taking the imaginary part on both sides of the first equation in \eqref{smooth-sol}.
Hence, we omit the details.
\end{proof}

The following corollary is a direct consequence of Proposition \ref{2prop2}.

\begin{corollary}\label{prop4}Under the same conditions  on the weight functions $w(x)$ and $g(x)$, and the indices $r_1$, $p_1$, $p_2$ in Proposition \ref{prop1}, we have
\begin{align}\label{J6}\int^{+\infty}_2\frac{1}{t} \int_{\mathbb {R}}| u(\alpha v+\beta |u|^2)| w'\left(\frac{x}{\lambda_{1}(t)}\right)g\left(\frac{x}{\lambda_{2}(t)}\right) d xdt\leq C(\|u_0\|_{H^1(\mathbb{R})},\|v_0\|_{H^1(\mathbb{R})}).
\end{align}
\end{corollary}

\begin{proof} Corollary \ref{prop4} follows from  Proposition \ref{2prop2} and the fact  that $|u|<|\Re u|+|\Im u|$.
\end{proof}

\medskip

\begin{proposition}\label{prop2}Under the same conditions  on the weight functions $w(x)$ and $g(x)$, and the indices $r_1$, $p_1$, $p_2$ in Proposition \ref{prop1}, for the functional
\begin{equation*}
J_2(t)=\frac{\theta_2}{\eta(t)}\int_{\mathbb{R}}v^2(x,t)\,w(\frac{x}{\lambda_1(t)})g(\frac{x}{\lambda_2(t)})dx,\hskip15pt \theta_2\in \mathbb{R}^{+},
\end{equation*}
defined in \eqref{J2}, there exists a function $J_{2,int}(t)\in L^1(2,+\infty)$ such that
\begin{equation}\label{Equ2.60}
\begin{split}
\frac{3\theta_2}{t} \int_{\mathbb {R}} |\partial_xv|^2 w'\left(\frac{x}{\lambda_{1}(t)}\right)g\left(\frac{x}{\lambda_{2}(t)}\right) d x
=&\; -\frac{d}{d t} J_{2}(t)+J_{2,int}(t)\\
&\;+\frac{2\theta_2}{t}\int_{\mathbb{R}}\left(\frac{v^3}{3}-\gamma |u|^2v\right)w'(\frac{x}{\lambda_1(t)})g(\frac{x}{\lambda_2(t)})dx\\
&\;-\frac{2\theta_2\gamma}{\eta(t)}\int_{\mathbb{R}}|u|^2\partial_xvw(\frac{x}{\lambda_1(t)})g(\frac{x}{\lambda_2(t)})dx
\end{split}
\end{equation} holds for any $t\geq2$.
\end{proposition}

\begin{proof} Performing the derivative in the variable $t$ it follows that
\begin{equation}\label{a2}
\begin{split}
\frac{d}{dt} J_{2}(t)=&\; \underbrace{\frac{-\theta_2 \eta^{'}(t)}{(\eta(t))^2}\int_{\mathbb{R}}v^2w(\frac{x}{\lambda_1(t)})g(\frac{x}{\lambda_2(t)})dx}_{J_{2,1}(t)}+\underbrace{\frac{\theta_2}{\eta(t)}\int_{\mathbb{R}}v^2\partial_t(w(\frac{x}{\lambda_1(t)})g(\frac{x}{\lambda_2(t)}))dx}_{J_{2,2}(t)}\\
&\;+\frac{2\theta_2}{\eta(t)}\int_{\mathbb{R}}v\partial_tvw(\frac{x}{\lambda_1(t)})g(\frac{x}{\lambda_2(t)})dx.
\end{split}
\end{equation}

Using  $\frac{\eta^{'}(t)}{\eta(t)}\backsim\frac{\lambda_1^{'}(t)}{\lambda_1(t)}\backsim\frac{\lambda_2^{'}(t)}{\lambda_2(t)}\backsim \frac{1}{t} $ and Lemma \ref{result6} we can easily obtain
\begin{equation*}
J_{2,1}(t), J_{2,2}(t)\in L^1(2,+\infty).
\end{equation*}

Next, we focus on the last term in the right hand side of \eqref{a2}.
Using the equation
\begin{equation*}
\partial_tv=-\partial^3_xv-v\partial_xv+\gamma \partial_x(|u|^2),
\end{equation*}
integrating by parts and apply Lemma \ref{result6} for several times, it follows
\begin{equation}\label{a3}
\begin{split}
\frac{2\theta_2}{\eta(t)}&\int_{\mathbb{R}}v\partial_tvw(\frac{x}{\lambda_1(t)})g(\frac{x}{\lambda_2(t)})dx\\
=&\;\frac{2\theta_2}{\eta(t)}\int_{\mathbb{R}}(\frac{v^2}{2}-\gamma |u|^2)\partial_x[vw(\frac{x}{\lambda_1(t)})g(\frac{x}{\lambda_2(t)})]dx- \frac{2\theta_2}{\eta(t)}\int_{\mathbb{R}}v\partial^3_xvw(\frac{x}{\lambda_1(t)})g(\frac{x}{\lambda_2(t)})dx\\
=&\; \frac{2\theta_2}{\eta(t)}\int_{\mathbb{R}}(\frac{v^3}{2}-\gamma |u|^2v)\partial_x[w(\frac{x}{\lambda_1(t)})g(\frac{x}{\lambda_2(t)})]dx-\frac{2\theta_2}{\eta(t)}\int_{\mathbb{R}}v\partial^3_xvw(\frac{x}{\lambda_1(t)})g(\frac{x}{\lambda_2(t)})dx\\
&\;-\frac{2\theta_2\gamma}{\eta(x)}\int_{\mathbb{R}}|u|^2\partial_xvw(\frac{x}{\lambda_1(t)})g(\frac{x}{\lambda_2(t)})dx+\frac{2\theta_2}{\eta(x)}\int_{\mathbb{R}}\frac{v^2}{2}\partial_xvw(\frac{x}{\lambda_1(t)})g(\frac{x}{\lambda_2(t)})dx\\
=&\; \frac{2\theta_2}{\eta(t)}\int_{\mathbb{R}}(\frac{v^3}{3}-\gamma |u|^2v)\partial_x[w(\frac{x}{\lambda_1(t)})g(\frac{x}{\lambda_2(t)})]dx-\frac{2\theta_2}{\eta(t)}\int_{\mathbb{R}}v\partial^3_xvw(\frac{x}{\lambda_1(t)})g(\frac{x}{\lambda_2(t)})dx\\
&\;-\frac{2\theta_2\gamma}{\eta(x)}\int_{\mathbb{R}}|u|^2\partial_xvw(\frac{x}{\lambda_1(t)})g(\frac{x}{\lambda_2(t)})dx,
\end{split}
\end{equation}
where
\begin{equation*}
\begin{split}
\frac{2\theta_2}{\eta(t)}\!\!\int_{\mathbb{R}}(\frac{v^3}{3}-\gamma |u|^2v)\partial_x[w(\frac{x}{\lambda_1(t)})g(\frac{x}{\lambda_2(t)})]dx
&= \frac{2\theta_2}{t}\!\int_{\mathbb{R}}(\frac{v^3}{3}-\gamma |u|^2v)w'(\frac{x}{\lambda_1(t)})g(\frac{x}{\lambda_2(t)})dx\\
&\hskip10pt+\underbrace{\frac{2\theta_2}{\eta(t)\lambda_2(t)}\!\!\int_{\mathbb{R}}(\frac{v^3}{3}-\gamma|u|^2v)w(\frac{x}{\lambda_1(t)})g^{'}(\frac{x}{\lambda_2(t)})dx}_{J_{2,3}(t)},
\end{split}
\end{equation*}
and
\begin{equation*}
\begin{split}
-\frac{2\theta_2}{\eta(t)}\int_{\mathbb{R}}&v\partial^3_xvw(\frac{x}{\lambda_1(t)})g(\frac{x}{\lambda_2(t)})dx\\
=&\underbrace{-\frac{3\theta_2}{\eta(t)\lambda_2(t)}\int_{\mathbb{R}}|\partial_xv|^2w(\frac{x}{\lambda_1(t)})g^{'}(\frac{x}{\lambda_2(t)})dx-\frac{2\theta_2}{\eta(t)}\int_{\mathbb{R}}v\partial_x v\partial^2_x(w(\frac{x}{\lambda_1(t)})g(\frac{x}{\lambda_2(t)}))dx}_{J_{2,4}(t)}\\
&\;-\frac{3\theta_2}{t}\int_{\mathbb{R}}|\partial_xv|^2w'(\frac{x}{\lambda_1(t)})g(\frac{x}{\lambda_2(t)})dx,
\end{split}
\end{equation*}
with\[J_{2,3}(t)\in L^1(2,+\infty), J_{2,4}(t)\in L^1(2,+\infty).\]

Substituting \eqref{a3} into \eqref{a2}, we immediately get the identity \eqref{Equ2.60} in Proposition \ref{prop2} with
\begin{equation*}
J_{2,int}(t)=J_{2,1}(t)+J_{2,2}(t)+J_{2,3}(t)+J_{2,4}(t)\in L^1(2,+\infty).
\end{equation*}
\end{proof}


In the next proposition we will analyze the following quantity
\begin{equation*}
\int_{\mathbb{R}}{\Im}[u(x,t)\partial_x\overline{u}(x,t)]w(\frac{x}{\lambda_1(t)})g(\frac{x}{\lambda_2(t)})dx
\end{equation*}
to derive an useful identity in our arguments. We shall mention that this functional is somehow related to the invariant \eqref{cons2}.

\begin{proposition}\label{prop3} Under the same conditions  on the weight functions $w(x)$ and $g(x)$, and the indices $r_1$,$p_1$,$p_2$ in Proposition \ref{prop1}, for the functional
\[ J_3(t)=\frac{\theta_3}{\eta(t)}\int_{\mathbb{R}}{\Im}[u(x,t)\partial_x\overline{u}(x,t)]w(\frac{x}{\lambda_1(t)})g(\frac{x}{\lambda_2(t)})dx,\quad \theta_3\in \mathbb{R}^{+},\]
defined in \eqref{J3} where $\Im (u(x,t))$ denotes the imaginary part of the function $u(x,t)$, there exists a function $J_{3,int}(t)\in L^1(2,+\infty)$ such that
\begin{equation}\label{Equ2.6}
\begin{split}
\frac{2\theta_3}{t} \int_{\mathbb {R}}& |\partial_xu|^2 w'\left(\frac{x}{\lambda_{1}(t)}\right)g\left(\frac{x}{\lambda_{2}(t)}\right) d x+ \frac{\beta \theta_3}{2t}\int_{\mathbb{R}}|u|^4w'(\frac{x}{\lambda_1(t)})g(\frac{x}{\lambda_2(t)})dx\\
=& -\frac{d}{d t} J_{3}(t)+J_{3,int}(t)+\frac{\theta_3\alpha}{\eta(t)}\int_{\mathbb{R}} |u|^2\partial_xvw(\frac{x}{\lambda_1(t)})g(\frac{x}{\lambda_2(t)})dx
\end{split}
\end{equation} holds for any $t\geq2$.
\end{proposition}

\begin{proof}
It can be directly calculated that
\begin{equation}\label{a4}
\begin{split}
\frac{d}{d t} J_{3}(t)=&\; \underbrace{\theta_3\int_{\mathbb{R}}{\Im }[u\partial_x\overline{u}]\partial_t\left[\frac{1}{\eta(t)}w(\frac{x}{\lambda_1(t)})g(\frac{x}{\lambda_2(t)})\right]dx }_{J_{3,1}(t)}\\
&\; +\underbrace{\frac{\theta_3}{\eta(t)}\int_{\mathbb{R}}\partial_t {\Im}[u\partial_x\overline{u}]w(\frac{x}{\lambda_1(t)})g(\frac{x}{\lambda_2(t)})dx}_{J_{3,2}(t)},
\end{split}
\end{equation}
where
\begin{equation}
\begin{split}
 \partial_t\left[\frac{1}{\eta(t)}w(\frac{x}{\lambda_1(t)})g(\frac{x}{\lambda_2(t)})\right]
 =&\;-\frac{\eta^{'}(t)}{\eta^2(t)}w(\frac{x}{\lambda_1(t)})g(\frac{x}{\lambda_2(t)})-\frac{\lambda_1^{'}(t)}{\eta(t)\lambda_1(t)}w'(\frac{x}{\lambda_1(t)})\frac{x}{\lambda_1(t)}g(\frac{x}{\lambda_2(t)})\\
 &\;-\frac{\lambda_2^{'}(t)}{\eta(t)\lambda_2(t)}w(\frac{x}{\lambda_1(t)})g^{'}(\frac{x}{\lambda_2(t)})\frac{x}{\lambda_2(t)}.
\end{split}
\end{equation}

Using $\frac{\eta^{'}(t)}{\eta(t)}\backsim\frac{\lambda_1^{'}(t)}{\lambda_1(t)}\backsim\frac{\lambda_2^{'}(t)}{\lambda_2(t)}\backsim \frac{1}{t} $, Lemma \ref{result6} and repeating the same techniques in \eqref{s-9}, one can directly check that \[J_{3,1}(t)\in L^1(2,+\infty).\]

Recalling the structure $i\partial_tu+\partial^2_xu=\alpha u v+\beta u|u|^2$ where $u(x,t)\in \mathbb{C}$, $v(x,t), \alpha, \beta \in \mathbb{R}$, then it can be directly computed that
\begin{equation}\label{bb2}
i\partial_tu\partial_x\overline{u}+\partial^2_xu\partial_x\overline{u}=\alpha u\partial_x\overline{u} v+\beta u\partial_x\overline{u}|u|^2,
\end{equation}
\begin{equation*}
-i\partial_t\overline{u}+\partial^2_x\overline{u}=\alpha \overline{u} v+\beta \overline{u}|u|^2,\]
and
\begin{equation}\label{bb3}
-iu\partial_t\partial_x\overline{u}+u\partial^3_x\overline{u}=\alpha u\partial_x(\overline{u} v)+\beta u \partial_x(\overline{u}|u|^2)=\alpha u\partial_x\overline{u} v+\alpha u\overline{u}\partial_xv+\beta u\partial_x\overline{u}|u|^2+\beta u\overline{u}\partial_x|u|^2.
\end{equation}
Now, subtracting  \eqref{bb3} from \eqref{bb2}, using the fact $u\overline{u}=|u|^2$, we have
\[i(\partial_tu\partial_x\overline{u}+u\partial_t\partial_x\overline{u})+(\partial^2_xu\partial_x\overline{u}-u\partial^3_x\overline{u})=-\alpha |u|^2\partial_xv-\beta |u|^2\partial_x|u|^2,\]i.e.,
\begin{equation}\label{bb4}
\partial_t(u\partial_x\overline{u})=(\partial_tu\partial_x\overline{u}+u\partial_t\partial_x\overline{u})=i(\partial^2_xu\partial_x\overline{u}-u\partial^3_x\overline{u})+i(\alpha |u|^2\partial_xv+\beta |u|^2\partial_x|u|^2).
\end{equation}
Since
\begin{equation*}
\Im [i(\partial^2_xu\partial_x\overline{u}-u\partial^3_x\overline{u})]=\Re(\partial^2_xu\partial_x\overline{u}-u\partial^3_x\overline{u}),\] \[\partial_x(|u|^2)=2\Re(u\partial_x\overline{u})=2\Re(\overline{u}\partial_xu),
\end{equation*}
and
\begin{align*}
\partial^2_x\Re (u\partial_x\overline{u})&=\Re [\partial^2_x(u\partial_x\overline{u})] \\
&=\Re (u\partial^3_x\overline{u})+2\Re(\partial_xu\partial^2_x\overline{u})+\Re(\partial_x\overline{u}\partial^2_xu)\\
&=\Re (u\partial^3_x\overline{u})+\frac{3}{2}\partial_x(|\partial_xu|^2),
\end{align*}
after taking the imaginary part on both  sides of \eqref{bb4}, it follows
\begin{equation}\label{bb5}
\begin{split}
\partial_t\Im(u\partial_x\overline{u})&=\Re(-u\partial^3_x\overline{u})+\frac{\partial_x(|\partial_xu|^2 )}{2}+\alpha|u|^2\partial_xv+\beta |u|^2\partial_x|u|^2\nonumber\\
&=-\partial^2_x\Re (u\partial_x\overline{u})+2\partial_x(|\partial_xu|^2)+\alpha|u|^2\partial_xv
+ \frac{\beta}{2} \partial_x|u|^4.
\end{split}
\end{equation}
Next, substituting the identity of $\partial_t\Im[u\partial_x\overline{u}]$ into $J_{3,2}(t)$, using integration by parts formula, and the fact $\eta(t)\lambda_1(t)=t$ for several times, we get
\begin{equation}
\begin{split}
J_{3,2}(t)=&\frac{\theta_3}{\eta(t)}\int_{\mathbb{R}}\partial_t {\Im}[u\partial_x\overline{u}]w(\frac{x}{\lambda_1(t)})g(\frac{x}{\lambda_2(t)})dx\\
=&-\frac{2\theta_3}{t}\int_{\mathbb{R}}|\partial_xu|^2w'(\frac{x}{\lambda_1(t)})g(\frac{x}{\lambda_2(t)})dx+\frac{\theta_3 \alpha}{\eta(t)}\int_{\mathbb{R}}|u|^2\partial_xvw(\frac{x}{\lambda_1(t)})g(\frac{x}{\lambda_2(t)})dx\\
 &\; -\frac{\beta \theta_3}{2t}\int_{\mathbb{R}}|u|^4w'(\frac{x}{\lambda_1(t)})g(\frac{x}{\lambda_2(t)})dx\\
 &\;\underbrace{-\frac{2\theta_3}{\eta(t)\lambda_2(t)}\int_{\mathbb{R}}|\partial_xu|^2w(\frac{x}{\lambda_1(t)})g^{'}(\frac{x}{\lambda_2(t)})dx}_{J_{3,21}(t)}\\
 &\;\underbrace{-\frac{\theta_3}{\eta(t)}\int_{\mathbb{R}}\text{Re}[u\partial_x\overline{u}]\partial^2_x\left[w(\frac{x}{\lambda_1(t)})g(\frac{x}{\lambda_2(t)})\right]dx}_{J_{3,22}(t)}\\
 &\;\underbrace{-\frac{\beta\theta_3}{2\eta(t)\lambda_2(t)}\int_{\mathbb{R}}|u|^4w(\frac{x}{\lambda_1(t)})g^{'}(\frac{x}{\lambda_2(t)})dx}_{J_{3,23}(t)},
\end{split}
\end{equation}
with
\begin{equation*}
J_{3,21}(t)\in L^1(2,+\infty),\quad J_{3,22}(t)\in L^1(2,+\infty),\quad J_{3,23}(t)\in L^1(2,+\infty).
\end{equation*}
Finally, substituting the identities of $J_{3,1}(t)$
and $J_{3,2}(t)$ into \eqref{a4}, we get the identity \eqref{Equ2.6} in Proposition \ref{prop3} with $J_{3,int}=J_{3,1}(t)+J_{3,21}(t)+J_{3,22}(t)+J_{3,23}(t)\in L^1(2,+\infty)$.
\end{proof}

\section{Proof of Theorems \ref{th1} and \ref{th2}}
\hspace*{\parindent}  With Propositions \ref{prop1}, \ref{prop2} and \ref{prop3} in hand, we are ready to prove Theorems \ref{th1} and \ref{th2}.

\medskip

\noindent\textbf{Proof of Theorem \ref{th1}:}
Since $w'(x)=g(x)\backsim e^{-|x|}$, $\lambda_1(t)=t^{p_1}$ and $\lambda_2(t)=t^{p_1p_2}$ with $p_2> 1$,  it follows from Proposition \ref{prop1} and Corollary \ref{prop4} that
 \begin{equation}\label{Equ3.7}
\begin{split}
\int^{\infty}_{2} &\frac{1}{t}\int_{|x|\lesssim t^{p_1}} |\frac{v^2}{2}-\gamma|u|^2| d x d t <\infty,
\end{split}
\end{equation}and
\[\int^{+\infty}_2\frac{1}{t} \int_{|x|\lesssim t^{p_1}}| u(\alpha v+\beta |u|^2)| d xdt<\infty,
\]respectively.

Due to $\frac{1}{t}\notin L^{1}(2,+\infty)$, there exists a positive time sequence $(t_{n})_{n\geq1}$ such that
\begin{equation*}
\begin{split}
\lim _{t_n\rightarrow \infty} \inf \int_{|x|\lesssim t^{p_1}_n} |\frac{v^2}{2}-\gamma|u|^2| d x=\lim _{t_n\rightarrow \infty} \inf \int_{|x|\lesssim t^{p_1}_n} | u(\alpha v+\beta |u|^2)|d x=0,
\end{split}
\end{equation*}
which implies Theorem \ref{th1}.
\qed

\medskip

\noindent\textbf{Proof of Theorem \ref{th2}}: Adding \eqref{Equ2.60} in Proposition \ref{prop2} and \eqref{Equ2.6} in Proposition \ref{prop3} together, we obtain the following identity which is essential in our analysis.

 \begin{equation}\label{bb7}
 \begin{split}
\frac{3\theta_2}{t} &\int_{\mathbb {R}}|\partial_xv|^2 w'(\frac{x}{\lambda_1(t)})g(\frac{x}{\lambda_2(t)}) d x+\frac{2\theta_3}{t} \int_{\mathbb {R}}|\partial_xu|^2 w'(\frac{x}{\lambda_1(t)})g(\frac{x}{\lambda_2(t)}) d x \\
=&\;-\frac{d}{dt}(J_2(t)+J_3(t))+J_{2,int}(t)+J_{3,int}(t)+\frac{2\theta_2}{t}\int_{\mathbb{R}}(\frac{v^3}{3}-\gamma |u|^2v)w'(\frac{x}{\lambda_1(t)})g(\frac{x}{\lambda_2(t)})dx\\
&\;-\frac{\beta\theta_3}{2t}\int_{\mathbb{R}}|u|^4w'(\frac{x}{\lambda_1(t)})g(\frac{x}{\lambda_2(t)})dx+(-\frac{2\theta_2\gamma}{\eta(t)}+\frac{\theta_3\alpha}{\eta(t)})\int_{\mathbb{R}} |u|^2\partial_xvw(\frac{x}{\lambda_1(t)})g(\frac{x}{\lambda_2(t)})dx.
\end{split}
\end{equation}

We first notice that the last term in the right hand side introduces an extra difficulty. The term as it stands cannot be controlled. For instance, integration by parts as usual does not work because the nice properties satisfied for the weight functions can be lost.

\medskip

To eliminate this difficult term we choose,
\begin{equation}\label{bb6}
\theta_2>0 \hskip10pt \text{and} \hskip10pt \theta_3=\frac{2\theta_2\gamma}{\alpha}.
\end{equation}

With this choice it follows from \eqref{bb7} that
\begin{equation}\label{bb8}
\begin{split}
&\frac{3\theta_2}{t} \int_{\mathbb {R}}|\partial_xv|^2 w'(\frac{x}{\lambda_1(t)})g(\frac{x}{\lambda_2(t)}) d x+\frac{2\theta_3}{t} \int_{\mathbb {R}}|\partial_xu|^2 w'(\frac{x}{\lambda_1(t)})g(\frac{x}{\lambda_2(t)}) d x \\
=&-\frac{d}{dt}(J_2(t)+J_3(t))+J_{2,int}(t)+J_{3,int}(t)+\frac{2\theta_2}{t}\int_{\mathbb{R}}(\frac{v^3}{3}-\gamma |u|^2v)w'(\frac{x}{\lambda_1(t)})g(\frac{x}{\lambda_2(t)})dx\\
&-\frac{\beta\theta_3}{2t}\int_{\mathbb{R}}|u|^4w'(\frac{x}{\lambda_1(t)})g(\frac{x}{\lambda_2(t)})dx.
\end{split}
\end{equation}


By using Lemma \ref{result6}, the Sobolev embedding and Proposition \ref{prop1},  we deduce the following chain of
inequalities
\begin{equation}\label{bb9}
\begin{split}
\int^{+\infty}_2\frac{1}{t}&\int_{\mathbb{R}}|v(\frac{v^2}{2}-\gamma |u|^2)|\;w'(\frac{x}{\lambda_1(t)})g(\frac{x}{\lambda_2(t)})dxdt\\
\leq&\|v\|_{L^{\infty}_{x,t}}\int^{+\infty}_2\frac{1}{t}\int_{\mathbb{R}}\big|\frac{v^2}{2}-\gamma |u|^2\big|w'(\frac{x}{\lambda_1(t)})g(\frac{x}{\lambda_2(t)})dxdt\\
\leq & \;C(\|u_0\|_{H^1(\mathbb{R})}, \|v_0\|_{H^1(\mathbb{R})} )\int^{+\infty}_2\frac{1}{t}\int_{\mathbb{R}}
|\frac{v^2}{2}-\gamma |u|^2|\,w'(\frac{x}{\lambda_1(t)})g(\frac{x}{\lambda_2(t)})dxdt\\
< & +\infty.
\end{split}
\end{equation}

On the other hand, combining Lemma \ref{result6}, the Sobolev embedding, and Corollary \ref{prop4}, it  leads to
\begin{equation}\label{bb10}
\begin{split}
\int^{+\infty}_2\frac{1}{t}&\int_{\mathbb{R}}|u|^2|\alpha v+\beta |u|^2|\;w'(\frac{x}{\lambda_1(t)})g(\frac{x}{\lambda_2(t)})dxdt\\
\leq&\;C(\|u_0\|_{H^1(\mathbb{R})}, \|v_0\|_{H^1(\mathbb{R})} )\int^{+\infty}_2\frac{1}{t}\int_{\mathbb{R}}|u(\alpha v+\beta |u|^2)|w'(\frac{x}{\lambda_1(t)})g(\frac{x}{\lambda_2(t)})dxdt\\
< &+\infty.
\end{split}
\end{equation}

The following identity plays a key role to establish the desired estimate,
\begin{equation}\label{cc1}
\frac{v^3}{3}-\gamma |u|^2v=\frac{2v}{3}(\frac{v^2}{2}-\gamma |u|^2)-\frac{\gamma |u|^2}{3\alpha}(\alpha v+\beta |u|^2)+\frac{\gamma \beta |u|^4}{3\alpha}.
\end{equation}

We first substitute \eqref{cc1} into \eqref{bb8}, then integrating over $(2,+\infty)$ with respect to the time variable
on both sides of \eqref{bb8} and using $\theta_3=\frac{2\theta_2\gamma}{\alpha}$, we obtain
\begin{equation}\label{bb8-b}
\begin{split}
\int^{+\infty}_{2}\frac{1}{t} &\int_{\mathbb {R}}\left(3\theta_2|\partial_xv|^2+\frac{4\theta_2\gamma}{\alpha}|\partial_xu|^2\right) w'(\frac{x}{\lambda_1(t)})g(\frac{x}{\lambda_2(t)}) d xdt\\
=&-(J_2(t)+J_3(t))+\int^{+\infty}_2(J_{2,int}(t)+J_{3,int}(t))\,dt\\
&+\frac{4\theta_2}{3t}\int^{+\infty}_2\frac{1}{t}\int_{\mathbb{R}}v(\frac{v^2}{2}-\gamma |u|^2)\;w'(\frac{x}{\lambda_1(t)})
g(\frac{x}{\lambda_2(t)})dxdt\\
&-\frac{2\gamma\theta_2}{3\alpha}\int^{+\infty}_2\frac{1}{t}\int_{\mathbb{R}}|u|^2(\alpha v+\beta |u|^2)\;w'(\frac{x}{\lambda_1(t)})g(\frac{x}{\lambda_2(t)})dxdt\\
&-\frac{\beta\gamma\theta_2}{3\alpha}\int_2^{\infty}\frac{1}{t}\int_{\mathbb{R}}|u|^4w'(\frac{x}{\lambda_1(t)})g(\frac{x}{\lambda_2(t)})dx.
\end{split}
\end{equation}

Taking into account that  $|J_{i}(t)|<\infty$, $J_{i,int}(t)\in L^1(2,+\infty)$, $i=2,3$, and the estimates \eqref{bb9} and \eqref{bb10}, we can deduce that there exists a constant $C>0$ such that
 \begin{equation}\label{m2}
 \begin{split}
\int^{+\infty}_{2}\frac{1}{t} &\int_{\mathbb {R}}\left(3\theta_2|\partial_xv|^2+\frac{4\theta_2\gamma}{\alpha}|\partial_xu|^2\right) w'(\frac{x}{\lambda_1(t)})g(\frac{x}{\lambda_2(t)}) d xdt \\
\leq &\;C+\frac{-\beta\theta_2\gamma}{3\alpha}\int^{+\infty}_{2}\frac{1}{t}\int_{\mathbb{R}}|u|^4w'(\frac{x}{\lambda_1(t)})g(\frac{x}{\lambda_2(t)})dxdt.
\end{split}
\end{equation}
The assumptions
\begin{equation*}
\label{m3}\beta>0,\quad \theta_2>0,\quad \theta_3>0,
\end{equation*}
allow us to conclude that
 \begin{equation}\label{m4}
 \int^{+\infty}_{2}\frac{1}{t} \int_{\mathbb {R}}\left(|\partial_xv|^2+ |\partial_xu|^2+|u|^4\right) w'(\frac{x}{\lambda_1(t)})g(\frac{x}{\lambda_2(t)}) d xdt <\infty.
 \end{equation}
This in turn implies \eqref{decay2} in Theorem \ref{th2} since $\frac{1}{t}\notin L^1(2,+\infty)$.

\medskip

Next, we treat \eqref{decay3} in Theorem \ref{th2}.  We  prove first that \eqref{decay3} holds for $p=4$ and $p=3$.
\medskip

Lemma \ref{result6}, the Sobolev embedding, and Proposition \ref{prop1} combined with the identity
$$
v^4=4\gamma^2|u|^4+4(\frac{v^2}{2}+\gamma |u|^2)(\frac{v^2}{2}-\gamma |u|^2)
$$
and \eqref{m4}, lead to
\begin{equation}\label{b1}
\int^{\infty}_{2}\frac{1}{t}\int_{\mathbb{R}}(|u|^4+v^4) w'(\frac{x}{\lambda_1(t)})g(\frac{x}{\lambda_2(t)}) d xdt <\infty.\end{equation}

On the other hand,  the identity
\begin{equation*}
|u|^3=-\frac{1}{\gamma}|u|(\frac{v^2}{2}-\gamma |u|^2)+\frac{v}{2\gamma \alpha}|u|(\alpha v+\beta |u|^2)-\frac{\beta}{2\gamma \alpha}v|u|^3,
\end{equation*}
and the inequality  $v|u|^3\leq v^4+|u|^4$ combined with Proposition \ref{prop1} and Corollary \ref{prop4}, yield
\begin{equation}\label{b2}
\int^{\infty}_{2}\frac{1}{t}\int_{\mathbb{R}}|u|^3 w'(\frac{x}{\lambda_1(t)})g(\frac{x}{\lambda_2(t)}) \, dxdt <\infty.\end{equation}
This in turn implies that
\begin{equation}\label{b3}
\int^{\infty}_{2}\frac{1}{t} \int_{\mathbb{R}}|uv| w'(\frac{x}{\lambda_1(t)})g(\frac{x}{\lambda_2(t)}) d xdt <\infty,\end{equation}
since $uv=\alpha^{-1}u(\alpha v+\beta |u|^2)-\alpha^{-1}\beta u |u|^2$ combined with Corollary \ref{prop4}
and \eqref{b2}.

Given $m\in \mathbb{R}^{+}$, we obtain the following elementary inequalities,
\begin{equation}\label{ineq-new-1}
\begin{split}
|v|^{2+m}&=2|v|^m(\frac{|v|^2}{2}-\gamma |u|^2)+2\gamma |v|^m|u|^2\\
&\leq 2\sup_{x,t}|v|^m |\frac{v^2}{2}-\gamma |u|^2|+2\gamma(\epsilon^{\frac{2+m}{m}}|v|^{2+m}+\frac{1}{\epsilon^{\frac{2+m}{2}}}|u|^{2+m})
\end{split}
\end{equation}
and
\begin{equation}\label{ineq-new-2}
\begin{split}
|u|^{2+m}&=\frac{-|u|^m}{\gamma}(\frac{|v|^2}{2}-\gamma |u|^2)+\frac{|u|^m|v|^2}{2\gamma} \nonumber\\
&\leq \frac{\sup_{x,t}|u|^m}{|\gamma|} |\frac{v^2}{2}-\gamma |u|^2|+\frac{1}{2\gamma}(\epsilon^{\frac{2+m}{m}}|u|^{2+m}+\frac{1}{\epsilon^{\frac{2+m}{2}}}|v|^{2+m}).
\end{split}
\end{equation}
Now choosing a sufficiently small $\epsilon>0$ in the above estimates and employing
Lemma \ref{result6} and Proposition \ref{prop1}, we infer that
\begin{equation*}
\int^{\infty}_{2}\frac{1}{t}\int_{\mathbb{R}}|u|^{2+m} w'(\frac{x}{\lambda_1(t)})g(\frac{x}{\lambda_2(t)})\, dxdt <\infty
\hskip5pt\text{and}\hskip5pt
\int^{\infty}_{2}\frac{1}{t}\int_{\mathbb{R}}|v|^{2+m} w'(\frac{x}{\lambda_1(t)})g(\frac{x}{\lambda_2(t)})\, dxdt <\infty
\end{equation*}
are equivalent, i.e.,
 \begin{equation}\label{equi1}
 \begin{split}
 \int^{\infty}_{2}\frac{1}{t}&\int_{\mathbb{R}}|u|^{2+m} w'(\frac{x}{\lambda_1(t)})g(\frac{x}{\lambda_2(t)})\, dxdt <\infty\\
 \hskip5pt\text{if and only if}\hskip40pt&\\
 \int^{\infty}_{2}\frac{1}{t}&\int_{\mathbb{R}}|v|^{2+m} w'(\frac{x}{\lambda_1(t)})g(\frac{x}{\lambda_2(t)})\, dxdt <\infty,
 \end{split}
 \end{equation}
 for any given $m\in \mathbb{R}^{+}$.

 \medskip

As a consequence of \eqref{b2}, combined with \eqref{equi1} and Lemma \ref{result6}, to prove \eqref{decay3} in Theorem \ref{th2}, we only need to show one of the following estimates
\begin{equation}\label{need1}
\int^{\infty}_{2}\frac{1}{t}\int_{\mathbb{R}}|u|^{2+\rho} w'(\frac{x}{\lambda_1(t)})g(\frac{x}{\lambda_2(t)}) d xdt <\infty,
\hskip15pt \forall \rho\in (0,1),
\end{equation}
\begin{equation}\label{need2}
\int^{\infty}_{2}\frac{1}{t}\int_{\mathbb{R}}|v|^{2+\rho} w'(\frac{x}{\lambda_1(t)})g(\frac{x}{\lambda_2(t)}) d xdt <\infty,\hskip15pt\forall \rho\in (0,1).
\end{equation}

Indeed, let us fix $\frac{1}{2}\leq \rho_1 \leq 1$. The Young inequality yields
$|v|^{\rho_1}|u|^2\lesssim(|v|^{\rho_1}|u|^{\rho_1})^{\frac{1}{\rho_1}}+|u|^{\frac{2-\rho_1}{1-\rho_1}}$ with
$\frac{2-\rho_1}{1-\rho_1}\geq 3$. It follows that
\begin{equation*}
\begin{split}
|v|^{2+\rho_1} &=2|v|^{\rho_1}(\frac{v^2}{2}-\gamma |u|^2)+2\gamma |v|^{\rho_1}|u|^2\\
& \lesssim  |v|^{\rho_1}|\frac{v^2}{2}-\gamma |u|^2|+ |uv|+|u|^{\frac{2-\rho_1}{1-\rho_1}}\\
& \lesssim  |\frac{v^2}{2}-\gamma |u|^2|+ |uv|+|u|^{3}.
\end{split}
\end{equation*}
Proposition \ref{prop1}, \eqref{b2}, and \eqref{b3} lead to
\begin{equation}\label{b4}
\int^{\infty}_{2}\frac{1}{t}\int_{\mathbb{R}}|v|^{2+\rho_1} w'(\frac{x}{\lambda_1(t)})g(\frac{x}{\lambda_2(t)}) d xdt <\infty, \hskip10pt\text{for all}\hskip5pt\frac{1}{2}\leq \rho_1\leq 1.
\end{equation}

\medskip

Now fixing $\widetilde{\rho}_1$ with $ 1-\frac{1}{1+\rho_1}\leq \widetilde{\rho}_1 < 1$. Applying once more the
Young inequality we obtain
\begin{equation*}
|v|^{\widetilde{\rho}_1}|u|^2\lesssim(|v|^{\widetilde{\rho}_1}|u|^{\widetilde{\rho}_1})^{\frac{1}{\widetilde{\rho}_1}}
+|u|^{\frac{2-\widetilde{\rho}_1}{1-\widetilde{\rho}_1}} \hskip10pt\text{for}\hskip5pt
\frac{2-\widetilde{\rho}_1}{1-\widetilde{\rho}_1}\geq 2+\rho_1,
\end{equation*}
Hence
\begin{equation*}
\begin{split}
|u|^{2+\widetilde{\rho}_1} &=-\frac{|u|^{\widetilde{\rho}_1}}{\gamma}(\frac{v^2}{2}-\gamma |u|^2)+\frac{|u|^{\widetilde{\rho}_1}|v|^2}{2\gamma} \\
& \lesssim  |\frac{v^2}{2}-\gamma |u|^2|+ |uv|+|v|^{\frac{2-\widetilde{\rho}_1}{1-\widetilde{\rho}_1}}\\
&\lesssim |\frac{v^2}{2}-\gamma |u|^2|+ |uv|+|v|^{2+\rho_1}.
\end{split}
\end{equation*}
Therefore, Proposition \ref{prop1}, \eqref{b3} and \eqref{b4} give us
\begin{equation}\label{b5}
\int^{\infty}_{1}\frac{1}{t}\int_{\mathbb{R}}|u|^{2+\widetilde{\rho}_1} w'(\frac{x}{\lambda_1(t)})
g(\frac{x}{\lambda_2(t)}) d xdt <\infty, \hskip10pt\text{for all}\hskip5pt\widetilde{\rho}_1\in[ 1-\frac{1}{1+\rho_1}, 1].\end{equation}

Repeating the previous arguments to get \eqref{b4} and \eqref{b5}, then using \eqref{equi1} for several times, for any $\rho_n\in (0,1)$ and $\widetilde{\rho}_n\in (0,1)$ with $n=2,3,4,\dots$, satisfying
\begin{equation*}
1-\frac{1}{1+\widetilde{\rho}_{n-1}} \leq  \rho_n< \rho_{n-1} \hskip10pt\text{and}\hskip10pt
1-\frac{1}{1+\rho_{n}} \leq  \widetilde{\rho}_n< \widetilde{\rho}_{n-1},
\end{equation*}
we conclude that
\begin{equation}\label{a5}
\int^{\infty}_{2}\frac{1}{t}\int_{\mathbb{R}}|v|^{2+\rho_n} w'(\frac{x}{\lambda_1(t)})g(\frac{x}{\lambda_2(t)}) d xdt <\infty, \quad \rho_n \searrow 0\quad \text{as } n\rightarrow +\infty,\footnote{We note that although $\lim_{n\rightarrow \infty}\rho_n=0$, we cannot let $\rho_n$ approximate to zero on both sides of \eqref{a5}.}
\end{equation}
and
\begin{equation*}
\int^{\infty}_{2}\frac{1}{t}\int_{\mathbb{R}}|u|^{2+\widetilde{\rho}_{n}} w'(\frac{x}{\lambda_1(t)})g(\frac{x}{\lambda_2(t)}) d xdt <\infty, \quad \widetilde{\rho}_n \searrow 0 \quad \text{as } n\rightarrow +\infty,
\end{equation*}
which implies \eqref{need1} and \eqref{need2}. Thus, we have \eqref{decay3}.
$\hfill{} \Box$

\section{Proof of Theorem \ref{th3}}

\hspace*{\parindent} We notice that \eqref{m2} does not directly imply \eqref{m4} if $\beta<0$. It means that
the argument to establish Theorem \ref{th2} does not work anymore to show Theorem \ref{th3}.
However, with Propositions \ref{prop1}, \ref{prop2}, and \ref{prop3}, and Corollary \ref{prop4} in hand, we can obtain
Theorem \ref{th3} under some \lq\lq smallness" assumption on the product of $-\beta$ and
$\Phi(\|u_0\|_{H^1(\mathbb{R})},\|v_0\|_{H^1(\mathbb{R})})$.

\medskip

\noindent\textbf{Proof of Theorem \ref{th3}}: We first recall
\eqref{1result} in Proposition \ref{prop1}, i.e.,
 \begin{align}\label{sm1result}
 &\int^{+\infty}_2\frac{1}{t}\int_{\mathbb{R}}|\frac{v^2}{2}-\gamma|u|^2|w'(\frac{x}{\lambda_1(t)})g(\frac{x}{\lambda_2(t)})dxdt\leq C(\|u_{0}\|_{H^1(\mathbb{R})},\|v_{0}\|_{H^1(\mathbb{R})}),
\end{align}
 and \eqref{J6} in Corollary \ref{prop4}, i.e.,
\begin{align}\label{smJ6}
\int^{+\infty}_2\frac{1}{t} \int_{\mathbb {R}}| u(\alpha v+\beta |u|^2)| w' \left(\frac{x}{\lambda_{1}(t)}\right)g\left(\frac{x}{\lambda_{2}(t)}\right) d xdt\leq C(\|u_0\|_{H^1(\mathbb{R})},\|v_0\|_{H^1(\mathbb{R})}).
\end{align}

We first observe that
\begin{equation*}
|\gamma|u|^4-\frac{v^2|u|^2}{2}|\leq \|u\|^2_{L^{\infty}(\mathbb{R}^{+};L^{\infty}(\mathbb{R}))} |\frac{v^2}{2}-\gamma|u|^2|\leq \|u\|^2_{C_tH^1_x}|\frac{v^2}{2}-\gamma|u|^2|,
\end{equation*}
and
\begin{equation*}
\alpha v|u|^2+\beta |u|^4| \leq \|\overline{u}\|_{_{L^{\infty}(\mathbb{R}^{+};L^{\infty}(\mathbb{R}))}} | u(\alpha v+\beta |u|^2)| \leq \|\overline{u}\|_{C_tH^1_x}| u(\alpha v+\beta |u|^2)|.
\end{equation*}
Then by Lemma \ref{result6} combined  with \eqref{sm1result} and \eqref{smJ6} it immediately follows that
\begin{equation}\label{c4}
I_1:=\int^{+\infty}_2\frac{1}{t}\int_{\mathbb{R}}|\gamma|u|^4-\frac{v^2|u|^2}{2}|w'(\frac{x}{\lambda_1(t)})g(\frac{x}{\lambda_2(t)})dxdt<C(\|u_{0}\|_{H^1(\mathbb{R})},\|v_{0}\|_{H^1(\mathbb{R})}),
\end{equation}
 and
\begin{equation}\label{c5}
I_2:=\int^{+\infty}_2 \frac{1}{t} \int_{\mathbb {R}}\big| \alpha v|u|^2+\beta |u|^4\big|w'\left(\frac{x}{\lambda_{1}(t)}\right)g\left(\frac{x}{\lambda_{2}(t)}\right) d xdt\leq C(\|u_0\|_{H^1(\mathbb{R})},\|v_0\|_{H^1(\mathbb{R})}).
\end{equation}

\medskip

 Next, operating $-\beta I_1+\gamma I_2$, we obtain from \eqref{c4} and \eqref{c5} that
\begin{equation}\label{salvo}
\begin{split}
\int^{+\infty}_2\frac{1}{t}\int_{\mathbb{R}}\Big|\alpha\gamma v|u|^2+\frac{v^2|u|^2\beta}{2}\Big| w'(\frac{x}{\lambda_1(t)})g(\frac{x}{\lambda_2(t)})dxdt
\leq &-\beta I_1+ \gamma I_2\\
\leq & \; C(\|u_0\|_{H^1(\mathbb{R})},\|v_0\|_{H^1(\mathbb{R})})
\end{split}
\end{equation}
since $\beta<0$ and $\gamma>0$.


Rewriting the first integral on the left hand side of \eqref{salvo} we have
\begin{equation}\label{c6}
\int^{+\infty}_2\int_{\mathbb{R}}\frac{1}{t}|u|^2|v||\alpha\gamma +\frac{v\beta}{2}|w'(\frac{x}{\lambda_1(t)})g(\frac{x}{\lambda_2(t)})dxdt<C(\|u_{0}\|_{H^1(\mathbb{R})},\|v_{0}\|_{H^1(\mathbb{R})}).
\end{equation}

To put forward the argument employed to prove Theorem \ref{th2} it was needed to establish the
following estimate
\begin{equation}\label{c6-b}
\int^{+\infty}_2\frac{1}{t}\int_{\mathbb{R}}|u|^2|v|w'(\frac{x}{\lambda_1(t)})g(\frac{x}{\lambda_2(t)})dxdt
<C(\|u_{0}\|_{H^1(\mathbb{R})},\|v_{0}\|_{H^1(\mathbb{R})}).
\end{equation}
This estimate allows us to obtain property \eqref{a0}, which is the key to derive the desired results.

\medskip

One way to do so is:  Given $\alpha>0$, $\gamma>0$ and $\beta<0$, choose a \lq\lq small" initial data
$(u_0,v_0)\in (H^1(\mathbb{R}))^2$ in the sense that
\begin{equation}\label{c12}
|\beta|\Phi \left(\left\|u_{0}\right\|_{H^{1}},\left\|v_{0}\right\|_{H^{1}}\right)\leq \alpha\gamma,
\end{equation}
where the constant $\Phi \left(\left\|u_{0}\right\|_{H^{1}},\left\|v_{0}\right\|_{H^{1}}\right)$ is the same as \eqref{Equ2.2} and $\underset{(x_1,x_2)\rightarrow (0,0)}{\lim} \Phi(x_1,x_2)=0$.  This implies that
\begin{equation}\label{c8}
\begin{split}
\alpha\gamma +\frac{v\beta}{2}&\geq \alpha\gamma -\frac{|\beta|\|v\|_{L^{\infty}_{x,t}}}{2}
\geq\alpha\gamma -\frac{|\beta|\|v\|_{C_tH^{1}_x}}{2}\\
&\geq \alpha\gamma -\frac{|\beta|\Phi \left(\left\|u_{0}\right\|_{H^{1}},\left\|v_{0}\right\|_{H^{1}}\right)}{2}
\geq \frac{\alpha\gamma}{2}.
\end{split}
\end{equation}

\begin{remark}
We need to emphasize that the constant $\Phi \left(\left\|u_{0}\right\|_{H^{1}},\left\|v_{0}\right\|_{H^{1}}\right)$ in \eqref{c12} also depends on the parameter vector $(\alpha,\beta,\gamma)\in \mathbb{R}^{+}\times \mathbb{R}^{-}\times \mathbb{R}^{+}$. We also note that Corcho-Linares in \cite{TAMS2007Linares} did not give the explicit dependence on the  constant $\Phi(\cdot,\cdot)$ because it was not necessary in their work.
\end{remark}

Suppose for a moment, that we have an explicit form of $\Phi \left(\left\|u_{0}\right\|_{H^{1}},\left\|v_{0}\right\|_{H^{1}}\right)$ as
\begin{equation*}
\Phi \left(\left\|u_{0}\right\|_{H^{1}},\left\|v_{0}\right\|_{H^{1}}\right)=\sqrt{2}C^{\frac{1}{2}}_{\alpha,\beta,\gamma}
(\|u_0\|_{H^1}+\|v_0\|_{H^1}+\|u_0\|^{5}_{H^1}+\|v_0\|^{5}_{H^1}).
\end{equation*}
with $C_{\alpha,\beta,\gamma}$ a constant given in \eqref{exp1} below which we will derive it in a while. Before that
we would like to present the argument to conclude the proof of Theorem \ref{th3}.

\medskip

Substituting \eqref{c8} into \eqref{c6}, it follows
\begin{equation}\label{c9}
\begin{split}
\int^{+\infty}_2\frac{1}{t}\int_{\mathbb{R}}|u|^2 |\alpha v|w'(\frac{x}{\lambda_1(t)})g(\frac{x}{\lambda_2(t)})dxdt<C(\|u_{0}\|_{H^1(\mathbb{R})},\|v_{0}\|_{H^1(\mathbb{R})}).
\end{split}
\end{equation}
Hence, combining \eqref{c5} and \eqref{c9}, we arrive at
\begin{equation}\label{c10}
\begin{split}
\int^{+\infty}_2\frac{1}{t}&\int_{\mathbb{R}}|\beta||u|^4 w'(\frac{x}{\lambda_1(t)})g(\frac{x}{\lambda_2(t)})dxdt\\
=&\int^{+\infty}_2\frac{1}{t}\int_{\mathbb{R}}||u|^2(\alpha v+\beta|u|^2)-\alpha|u|^2v |w'(\frac{x}{\lambda_1(t)})g(\frac{x}{\lambda_2(t)})dxdt\\
\leq &\int^{+\infty}_2\frac{1}{t}\int_{\mathbb{R}}||u|^2(\alpha v+\beta|u|^2)|w'(\frac{x}{\lambda_1(t)})g(\frac{x}{\lambda_2(t)})dxdt\\
&+\int^{+\infty}_2\frac{1}{t}\int_{\mathbb{R}}|u|^2|\alpha v |w'(\frac{x}{\lambda_1(t)})g(\frac{x}{\lambda_2(t)})dxdt\\
\leq&\;C(\|u_{0}\|_{H^1(\mathbb{R})},\|v_{0}\|_{H^1(\mathbb{R})}).
\end{split}
\end{equation}

Finally, substituting \eqref{c10} into \eqref{m2} and using the assumption $\theta_1>0$ and $\theta_2>0$, we arrive at the key estimate
 \begin{equation}\label{c11}
 \int^{+\infty}_{2} \frac{1}{t}\int_{\mathbb {R}}\left(|\partial_xv|^2+ |\partial_xu|^2+|u|^4\right) w'(\frac{x}{\lambda_1(t)})g(\frac{x}{\lambda_2(t)}) \,dxdt <\infty.
 \end{equation}

 We observe that the key estimate \eqref{c11} is exactly the same as in \eqref{m4}. Therefore, with \eqref{c11} at hand,  after repeating the argument from \eqref{b1} to \eqref{b5}, we can essentially complete the proof of Theorem \ref{th3}.

\medskip
From now on we derive an explicit form of $\Phi \left(\left\|u_{0}\right\|_{H^{1}},\left\|v_{0}\right\|_{H^{1}}\right)$. Indeed, using the three conserved quantities satisfy for the flow of  the Schr\"odinger-KdV system, i.e. \eqref{cons1}, \eqref{cons2}, and \eqref{cons3} combined with
the Cauchy-Schwarz inequality and the Gagliardo-Nirenberg inequality with the optimal constant $C_{GN}$, it can be directly checked that
\begin{equation}\label{apri3}
|\mathcal{M}(0)|=|\mathcal{M}(t)|=\|u_0\|^2,
\end{equation}
\begin{equation}\label{apri3-b}
|\mathcal{Q}(0)|\leq |\alpha|\|v_0\|^2+2|\gamma|\|u_0\|\|\partial_xu_0\|\leq (|\alpha|+2|\gamma|)\|u_0\|^2_{H^1},
\end{equation}
\begin{equation}\label{apri3-c}
\begin{split}
|\mathcal{E}(0)|&\leq |\alpha\gamma|\|v_0\|\|u_0\|^2_{L^4}+\frac{\alpha}{6} \|v_0\|^3+\frac{|\beta\gamma|}{2}\|u_0\|^4_{L^4}+\frac{|\alpha|}{2}\|\partial_xv_0\|^2+|\gamma|\|\partial_xu_0\|^2 \\
\leq &\; C_{GN}|\alpha\gamma|\|v_0\|\|u_0\|^{\frac{3}{2}}\|\partial_xu_0\|^{\frac{1}{2}}+\frac{|\alpha|}{6} \|v_0\|^3+\frac{|\beta\gamma|C_{GN}}{2}\|\partial_xu_0\|\|u_0\|^3\\
&+\frac{|\alpha|}{2}\|\partial_xv_0\|^2+|\gamma|\|\partial_xu_0\|^2\\
\leq &\; C_{GN}|\alpha\gamma|\|v_0\|_{H^1}\|u_0\|^2_{H^1}+\frac{|\alpha|}{6}\|v_0\|^3_{H^1}+\frac{|\beta\gamma|C_{GN}}{2}\|u_0\|^4_{H^1}+\frac{|\alpha|}{2}\|v_0\|^2_{H^1}+|\gamma|\|u_0\|^2_{H^1}\\
\leq &\; (C_{GN}|\alpha\gamma|+|\alpha|+|\gamma|) (\|u_0\|^2_{H^1}+\|v_0\|^2_{H^1}+\|v_0\|^4_{H^1}+\|u_0\|^4_{H^1})+\frac{|\beta\gamma|C_{GN}}{2}\|u_0\|^4_{H^1}.
\end{split}
\end{equation}

Moreover, from the definition of $\mathcal{Q}(t)$ and $\mathcal{M}(t)$, it can be directly checked that
\begin{align}\label{apri4}
\|v\|^2(t)&=|\frac{1}{\alpha}\mathcal{Q}(0)-\frac{2\gamma}{\alpha}\int_{\mathbb{R}}\Im (u\partial_x\overline{u})dx|
\leq \frac{1}{|\alpha|}\{|\mathcal{Q}(0)|+2|\gamma|\|u_0\|\|\partial_xu\|\}.
\end{align}
Now, we let
\begin{equation}\label{def1}
\mu:=\min\{|\gamma|,\frac{|\alpha|}{2}\}.
\end{equation}
Using the definition of $\mathcal{E}(t)$, the assumption $\alpha\gamma>0$ and the Cauchy-Schwarz inequality, it follows
\begin{align}\label{apri0}
\|\partial_xu\|^2&+\|\partial_xv\|^2\leq \frac{1}{\mu}|\{\gamma\|\partial_xu\|^2+\frac{\alpha}{2}
\|\partial_xv\|^2\}|\nonumber\\
&=\frac{1}{\mu}\big| \mathcal{E}(0)-\int_{\mathbb{R}}\big(\alpha \gamma v|u|^2-\frac{|\alpha|}{6}\,v^3
+\frac{|\beta\gamma}{2}|u|^4\big)\,dx\big|\nonumber\\
&\leq \frac{1}{\mu}\{|\mathcal{E}(0)|+|\alpha\gamma|\int_{\mathbb{R}}|v||u|^2dx+\frac{|\alpha|}{6}
\int_{\mathbb{R}}|v|^3dx+\frac{|\beta\gamma|}{2}\int_{\mathbb{R}}|u|^4dx\}\nonumber\\
&\leq \frac{1}{\mu}\{|\mathcal{E}(0)|+|\mathcal{Q}(0)|+2|\gamma|\|u_0\|\|\partial_xu\|+(|\alpha\gamma^2|+\frac{|\beta\gamma|}{2})\|u\|^4_{L^4}+\frac{|\alpha|}{6}\|v\|^3_{L^3}\}.
\end{align} Applying the Gagliardo-Nirenberg inequality with the optimal constant $C_{GN}$ and  the Young inequality for several times, for any $\epsilon_1$, $\epsilon_2$, $\epsilon_3$ and $\epsilon_4\in \mathbb{R}^{+}$, we have
\[\|u\|^4_{L^4}\leq (C_{GN}\|\partial_xu\|^{\frac{1}{4}}\|u\|^{\frac{3}{4}})^4\leq \epsilon_1\|\partial_xu\|^2+\frac{C^8_{GN}}{\epsilon_1}\|u_0\|^6,\quad \forall \epsilon_1>0.\]
Recalling the upper bound of $\|v\|(t)$, we derive the following inequalities
\begin{align*}
\|v\|^3_{L^3}&\leq (C_{GN}\|\partial_xv\|^{\frac{1}{6}}\|v\|^{\frac{5}{6}})^3
\leq \epsilon_2\|\partial_xv\|^2+\frac{C^{\frac{4}{3}}_{GN}}{\epsilon^{\frac{1}{3}}_2}\|v\|^{\frac{10}{3}}\\
&\leq
\epsilon_2\|\partial_xv\|^2+\frac{C^{\frac{4}{3}}_{GN}}{\epsilon^{\frac{1}{3}}_2|\alpha|^{\frac{5}{3}}}\{|\mathcal{Q}(0)|+2|\gamma|\|u_0\|\|\partial_xu\|\}^{\frac{5}{3}}\\
&\leq \epsilon_2\|\partial_xv\|^2+\frac{C^{\frac{4}{3}}_{GN}2^{\frac{5}{3}}|\mathcal{Q}(0)|^{\frac{5}{3}}}{\epsilon^{\frac{1}{3}}_2|\alpha|^{\frac{5}{3}}}+\frac{C^{\frac{4}{3}}_{GN}2^{\frac{5}{3}}}{\epsilon^{\frac{1}{3}}_2|\alpha|^{\frac{5}{3}}}(2|\gamma|\|u_0\|)^{\frac{5}{3}}\|\partial_xu\|^{\frac{5}{3}}\\
&\leq
\epsilon_2\|\partial_xv\|^2+\frac{C^{\frac{4}{3}}_{GN}2^{\frac{5}{3}}|\mathcal{Q}(0)|^{\frac{5}{3}}}{\epsilon^{\frac{1}{3}}_2|\alpha|^{\frac{5}{3}}}+\epsilon_3\|\partial_xu\|^2+\frac{1}{\epsilon^5_3}[\frac{C^{\frac{4}{3}}_{GN}2^{\frac{5}{3}}}{\epsilon^{\frac{1}{3}}_2|\alpha|^{\frac{5}{3}}}(2|\gamma|\|u_0\|)^{\frac{5}{3}}]^6\nonumber\\
&\leq
\epsilon_2\|\partial_xv\|^2+\frac{C^{\frac{4}{3}}_{GN}2^{\frac{5}{3}}|\mathcal{Q}(0)|^{\frac{5}{3}}}{\epsilon^{\frac{1}{3}}_2|\alpha|^{\frac{5}{3}}}+\epsilon_3\|\partial_xu\|^2+\frac{1}{\epsilon^5_3}\frac{C^{8}_{GN}2^{20}|\gamma|^{10}\|u_0\|^{10}}{\epsilon^{2}_2|\alpha|^{10}},\quad\forall \epsilon_2>0,\quad \forall\epsilon_3>0.
\end{align*}
\begin{equation*}
2\gamma \|u_0\|\|\partial_xu\|\leq \epsilon_4\|\partial_xu\|^2+\frac{4\gamma^2\|u_0\|^2}{\epsilon_4},\quad\quad\quad\forall \epsilon_4>0.
\end{equation*}
Next, plugging the estimates for $\|u\|^4_{L^4}$, $\|v\|^3_{L^3}$ and $2\gamma \|u_0\|\|\partial_xu\|$ into the
right-hand side of \eqref{apri0}, we arrive at
\begin{equation}\label{apri1}
\begin{split}
\|\partial_xu\|^2&+\|\partial_xv\|^2
\leq \frac{1}{\mu}\{|\mathcal{E}(0)|+|\mathcal{Q}(0)|+2|\gamma|\|u_0\|\|\partial_xu\|
+\big(|\alpha|\gamma^2+\frac{|\beta\gamma|}{2}\big)\|u\|^4_{L^4}\}\\
&\leq \underbrace{ \frac{|\mathcal{E}(0)|+|\mathcal{Q}(0)|+\frac{4\gamma^2\|u_0\|^2}{\epsilon_4}}{\mu}}_{A_1}+\frac{\epsilon_4\|\partial_xu\|^2}{\mu}+\frac{\epsilon_1}{\mu}(|\alpha\gamma^2|
+\frac{|\beta\gamma|}{2})\|\partial_xu\|^2\\
&\quad+\underbrace{\frac{1}{\epsilon_1\mu}(|\alpha\gamma^2|+\frac{|\beta\gamma|}{2})C^8_{GN}\|u_0\|^6}_{A_2}
+\frac{1}{\mu}\frac{|\alpha|}{6}(\epsilon_2\|\partial_xv\|^2+\epsilon_3\|\partial_xu\|^2)\\
&\quad+\underbrace{\frac{1}{\mu}\frac{|\alpha|}{6}\{\frac{C^{\frac{4}{3}}_{GN}2^{\frac{5}{3}}|\mathcal{Q}(0)|^{\frac{5}{3}}}{\epsilon^{\frac{1}{3}}_2|\alpha|^{\frac{5}{3}}}+\frac{1}{\epsilon^5_3}\frac{C^{24}_{GN}2^{20}|\gamma|^{10}\|u_0\|^{10}}{\epsilon^{2}_2|\alpha|^{10}}\}}_{A_3}, \quad \forall \epsilon_1,\epsilon_2,\epsilon_3,\epsilon_4\in \mathbb{R}^{+}.
\end{split}
\end{equation}
After choosing
\begin{equation*}\epsilon_1=\frac{\mu}{8(|\alpha\gamma^2|+\frac{|\beta\gamma|}{2})},\quad \epsilon_2=\epsilon_3=\frac{6\mu}{8|\alpha|},\quad \epsilon_4=\frac{\mu}{8},
\end{equation*}
in \eqref{apri1}, we have
\begin{align*}\|\partial_xu\|^2+\|\partial_xv\|^2
&\leq \frac{1}{2}(\|\partial_xu\|^2+\|\partial_xv\|^2)+A_1+A_2+A_3,
\end{align*}
i.e.,
\begin{align}\label{apri2}
\|\partial_xu\|^2+\|\partial_xv\|^2
&\leq 2(A_1+A_2+A_3),
\end{align}
where once again the use  of the upper bounds of $|\mathcal{M}(0)|$, $|\mathcal{Q}(0)|$ and $|\mathcal{E}(0)|$, lead to
\begin{align*}
A_1&= \frac{|\mathcal{E}(0)|+|\mathcal{Q}(0)|+\frac{4\gamma^2\|u_0\|^2}{\epsilon_4}}{\mu}\nonumber\\
&\leq \frac{(C_{GN}|\alpha\gamma|+2|\alpha|+3|\gamma|+\frac{32\gamma^2}{\mu})}{\mu} (\|u_0\|^2_{H^1}+\|v_0\|^2_{H^1}+\|v_0\|^4_{H^1}+\|u_0\|^4_{H^1})+\frac{|\beta\gamma|C_{GN}}{2\mu}\|u_0\|^4_{H^1}\\
&\leq \frac{(C_{GN}|\alpha\gamma|+2|\alpha|+3|\gamma|+\frac{32\gamma^2}{\mu}+\frac{|\beta\gamma|C_{GN}}{2})}{\mu} (\|u_0\|^2_{H^1}+\|v_0\|^2_{H^1}+\|v_0\|^{10}_{H^1}+\|u_0\|^{10}_{H^1}),
\end{align*}
\begin{equation*}
\begin{split}
A_2&=\frac{1}{\epsilon_1\mu}(|\alpha\gamma^2|+\frac{|\beta\gamma|}{2})C^8_{GN}\|u_0\|^6
\leq\frac{8(|\alpha\gamma^2|+\frac{|\beta\gamma|}{2})^2}{\mu^2}C^8_{GN}\|u_0\|^6_{H^1}\\
&\leq \frac{8(|\alpha\gamma^2|+\frac{|\beta\gamma|}{2})^2}{\mu^2}C^8_{GN}(\|u_0\|^2_{H^1}+\|v_0\|^2_{H^1}+\|v_0\|^{10}_{H^1}+\|u_0\|^{10}_{H^1}),
\end{split}
\end{equation*}
and
\begin{equation*}
\begin{split}
A_3&=\frac{1}{\mu}\frac{|\alpha|}{6}\{\frac{C^{\frac{4}{3}}_{GN}2^{\frac{5}{3}}|\mathcal{Q}(0)|^{\frac{5}{3}}}{\epsilon^{\frac{1}{3}}_2|\alpha|^{\frac{5}{3}}}+\frac{1}{\epsilon^5_3}\frac{C^{8}_{GN}2^{20}|\gamma|^{10}\|u_0\|^{10}}{\epsilon^{2}_2|\alpha|^{10}}\}\\
&\leq \frac{1}{\mu}\frac{C^{8}_{GN}2^{20}}{6\epsilon^{\frac{1}{3}}_2|\alpha|^{\frac{2}{3}}}\{|\mathcal{Q}(0)|^{\frac{5}{3}}+\frac{|\gamma|^{10}\|u_0\|^{10}}{\epsilon^{\frac{20}{3}}_2|\alpha|^{\frac{25}{3}}}\}\\
\end{split}
\end{equation*}
\begin{equation*}
\begin{split}
&\leq \frac{1}{\mu}\frac{C^{8}_{GN}2^{20}}{6(\frac{3\mu}{4})^{\frac{1}{3}}|\alpha|^{\frac{1}{3}}}\{(|\alpha|+2|\gamma|)^{\frac{5}{3}}\|u_0\|^{\frac{10}{3}}_{H^1}+\frac{|\gamma|^{10}\|u_0\|^{10}}{(\frac{3\mu}{4})^{\frac{20}{3}}|\alpha|^{\frac{5}{3}}}\}\\
&\leq \frac{C^{8}_{GN}2^{20}}{\mu^{\frac{4}{3}}|\alpha|^{\frac{1}{3}}}\{(|\alpha|+2|\gamma|)^{\frac{5}{3}}\|u_0\|^{\frac{10}{3}}_{H^1}+\frac{|\gamma|^{10}\|u_0\|^{10}}{\mu^{\frac{20}{3}}|\alpha|^{\frac{5}{3}}}\}\\
&\leq \frac{C^{8}_{GN}2^{20}}{\mu^{\frac{4}{3}}|\alpha|^{\frac{1}{3}}}\left[(|\alpha|+2|\gamma|)^{\frac{5}{3}}+\frac{|\gamma|^{10}}{\mu^{\frac{20}{3}}|\alpha|^{\frac{5}{3}}}\right](\|u_0\|^2_{H^1}+\|v_0\|^2_{H^1}+\|v_0\|^{10}_{H^1}+\|u_0\|^{10}_{H^1}).
\end{split}
\end{equation*}

Recalling \eqref{apri3} and \eqref{apri4},
we have
\begin{align*}
\|u\|^2(t)+\|v\|^2(t)&\leq \|u_0\|^2+\frac{1}{|\alpha|}\{|\mathcal{Q}(0)|+2|\gamma|\|u_0\|\|\partial_xu\|\}\nonumber\\
&\leq (2+\frac{2|\gamma|}{|\alpha|}+\frac{8\gamma^2}{|\alpha|^2})\|u_0\|^2_{H^1}+\frac12\|\partial_xu\|^2(t),
\hskip10pt \forall t\geq 0.\end{align*}
 Applying $\frac{2|\gamma|}{\alpha}\|u_0\|\|\partial_xu\|\leq \frac{\|\partial_xu\|^2}{2}+\frac{8r^2}{\alpha^2}\|u_0\|^2$ and \eqref{apri3-b}, it follows
\begin{equation}\label{apri5}
\|u\|^2(t)+\|v\|^2(t)-\frac12\|\partial_xu\|^2(t)\leq (2+\frac{2|\gamma|}{|\alpha|}+\frac{8\gamma^2}{|\alpha|^2})
(\|u_0\|^2_{H^1}+\|v_0\|^2_{H^1}),\hskip10pt \forall t\geq 0.
\end{equation}

Now, combining \eqref{apri2}, \eqref{apri5} and the definition of $\mu=\min\{|\gamma|,\frac{|\alpha|}{2}\}$, we get
\begin{equation}\label{equ2}
\|u\|^2(t)+\|v\|^2(t)+\frac{3}{2}\|\partial_xu\|^2(t)+2\|\partial_xv\|^2(t)\leq \widetilde{\Phi} \left(\left\|u_{0}\right\|_{H^{1}},\left\|v_{0}\right\|_{H^{1}}\right), \quad\quad \forall t\in \mathbb{R}^{+},
\end{equation}
where
\begin{equation*}
\widetilde{\Phi} \left(\left\|u_{0}\right\|_{H^{1}},\left\|v_{0}\right\|_{H^{1}}\right)=
C_{\alpha,\beta,\gamma} (\|u_0\|^2_{H^1}+\|v_0\|^2_{H^1}+\|u_0\|^{10}_{H^1}+\|v_0\|^{10}_{H^1})
\end{equation*}
 with
 \begin{equation}\label{exp1}
 \begin{split}
 C_{\alpha,\beta,\gamma}=&(2+\frac{2|\gamma|}{|\alpha|}+\frac{8\gamma^2}{|\alpha|^2})+\frac{4(C_{GN}|\alpha\gamma|+2|\alpha|+3|\gamma|+\frac{32\gamma^2}{\min\{|\gamma|,\frac{|\alpha|}{2}\}})+2|\beta\gamma|C_{GN}}{\min\{|\gamma|,\frac{|\alpha|}{2}\}}\\
 &+\frac{32(|\alpha\gamma^2|+\frac{|\beta\gamma|}{2})^2}{(\min\{|\gamma|,\frac{|\alpha|}{2}\})^2}C^8_{GN}\\
 &+\frac{C^{24}_{GN}2^{22}}{(\min\{|\gamma|,\frac{|\alpha|}{2}\})^{\frac{4}{3}}|\alpha|^{\frac{1}{3}}}\left[(|\alpha|+2|\gamma|)^{\frac{5}{3}}+\frac{|\gamma|^{10}}{(\min\{|\gamma|,\frac{|\alpha|}{2}\})^{\frac{20}{3}}|\alpha|^{\frac{5}{3}}}\right].\end{split}
 \end{equation}

Therefore, in \eqref{Equ2.2}, we have an explicit form of $\Phi \left(\left\|u_{0}\right\|_{H^{1}},\left\|v_{0}\right\|_{H^{1}}\right)$ as
\begin{equation}\label{phi}
\Phi \left(\left\|u_{0}\right\|_{H^{1}},\left\|v_{0}\right\|_{H^{1}}\right)=C^{\frac{1}{2}}_{\alpha,\beta,\gamma}
(\|u_0\|_{H^1}+\|v_0\|_{H^1}+\|u_0\|^{5}_{H^1}+\|v_0\|^{5}_{H^1}).
\end{equation}
$\hfill{} \Box$\\

\section{Acknowledgments}
 The first author was partially supported by CNPq grant 310329/2023-0 and FAPERJ grant E-26/200.465/2023.
 The second author is thankful to the Department of Applied Mathematics as well as  the Research
Center for Nonlinear Analysis at Hong Kong Polytechnic University for pleasant hospitality where part of this work was developed. The second author is  supported by NSFC (Grant No. 11901067), the Fundamental Research Funds for the Central Universities (Grant No. 2024IAIS-ZX001), and Key Laboratory of Nonlinear Analysis and its Applications (Chongqing University), Ministry of Education. The authors would like to thank A. Mendez for fruitful conversations concerning this work.


\begin{thebibliography}{99}
{\footnotesize

\bibitem{JDE2006Corcho}A. Arbieto, A. J. Corcho, C. Matheus, Rough solutions for the periodic Schr\"odinger-Korteweg-de Vries system, J. Differential Equations 230 (2006) 295--336.

\bibitem{OQE2021ASAO}L. Akinyemi, M. Senol, U. Akpan, K. Oluwasegun, The optical soliton solutions of generalized coupled nonlinear Schr\"odinger-Korteweg-de Vries equations, Optical and Quantum Electronics 53 (2021) 1--14.

\bibitem{Arxiv2024Ban} Y. Ban, J. Chen, Y. Zhang, Local well-posedness for the Schr\"odinger-KdV system in $H^{s_1}\times H^{s_2}$, arXiv:2411.10975.



\bibitem{JFA1998Ponce}   D. Bekiranov,  T. Ogawa, G. Ponce, Interaction equations for short and long dispersive waves, J. Funct. Anal. 158 (1998) 357--388.
\bibitem{PAMS1997Ponce}  D. Bekiranov,  T. Ogawa, G. Ponce, Weak solvability and well--posedness of a coupled
Schr\"odinger-Korteweg de Vries equation for capillary-gravity wave interactions, Proc. Amer.
Math. Soc. 125 (1997) 2907--2919.
\bibitem{AMS1999Bourgain}  J. Bourgain, Global solutions of nonlinear Schr\"odinger equations, American Mathematical Society, Providence, RI, 1999.

\bibitem{Arxiv2024Chenjie} J. Chen, F. Gu, B. Guo, Stochastic Schr\"odinger-Korteweg de Vries systems driven by multiplicative noises, arXiv:2402.04669v3.

\bibitem{TAMS2007Linares}A. J.  Corcho, F. Linares, Well--posedness for the Schr\"odinger--Korteweg--de Vries system,
Trans. Amer. Math. Soc. 359 (2007) 4089--4106.

\bibitem{Arxiv2024Linares}S. Correia, F. Linares, J. D. Silva, Sharp local well-posedness for the Schr\"odinger-Korteweg-de Vries system. arXiv:2408.10028.

\bibitem{SIAM2024SOS} S. Correia, F. Oliveira, J. D. Silva, Sharp local well-posedness and nonlinear smoothing for dispersive equations through frequency-restricted estimates,
SIAM J. Math. Anal. 56 (2024) 5604--5633.


\bibitem{2014FES} A. Filiz, M. Ekici, A. Sonmezoglu, $F$--Expansion Method and New Exact Solutions of the Schr\"odinger--KdV Equation, The Scientific World Journal, 2014.

\bibitem{JDE2010Guozihua} Z. Guo, Y. Wang, On the well-posedness of the Schr\"odinger--Korteweg--de Vries system,
J. Differential Equations 249 (2010) 2500--2520.

\bibitem{PRL1974Mima} H. Hojo, H. Ikezi, K. Mima, K. Nishikawa, Coupled nonlinear electron-plasma
and ion-acoustic waves, Phys. Rev. Lett. 33(1974) 148--151.





\bibitem{JPSJ1975Kawahara}   T. Kawahara, N. Sugimoto, T. Kakutani, Nonlinear interaction between short and long
capillary--gravity waves, J. Phys. Soc. Japan 39 (1975) 1379--1386.
\bibitem{CPAM1993KPV}   C. E. Kenig, G. Ponce, L. Vega, Well--posedness and scattering results for the generalized
Korteweg--de Vries equation via the contraction principle, Comm. Pure Appl. Math. 46 (1993) 527--620.
\bibitem{JAMS1996KPV}   C. E. Kenig, G. Ponce, L. Vega, A bilinear estimate with applications to the KdV
equation, J. Amer. Math. Soc. 9 (1996) 573--603.
\bibitem{AM2019Visan}  R. Killip, M. Visan, KdV is well-posed in $H^{-1}$, Ann. of Math. 190 (2019) 249--305.

\bibitem{MS2024Kumar} D. Kumar, A. Yildirim, M. K. Kaabar, H. Rezazadeh, M. E. Samei, Exploration of some novel solutions to a coupled Schrodinger-KdV equations in the interactions of capillary-gravity waves, Math. Sci. 18 (2024) 291--303.



\bibitem{M2009Linares}  F. Linares, G. Ponce, Introduction to Nonlinear Dispersive Equations, Springer, 2009.
\bibitem{N2019Linares}  F. Linares, J.M. Palacios, Dispersive blow-up and persistence properties for the
Schr\"odinge--Korteweg--de Vries system, Nonlinearity 32 (2019) 4996--5016.
\bibitem{SIAM2021Linares} F. Linares, A. J. Mendez, On long time behavior of solutions of the Schr\"odinger--Korteweg--de Vries system, SIAM J. Math. Anal. 53(2021) 3838-3855.


 \bibitem{CPDE2021MMPP} A. J. Mendez, C. Mu\~ noz, F. Poblete, J. C. Pozo, On local energy decay for large solutions of the Zakharov-Kuznetsov equation, Comm. Partial Differential Equations 46(2021) 1440-1487.

\bibitem{CMP2019Ponce} C. Mu\~ noz, G. Ponce, Breathers and the dynamics of solutions in KdV type equations, Comm. Math. Phys. 367(2019) 581-598.


\bibitem{DIE2005Pecher}  H. Pecher, The Cauchy problem for a Schr\"odinger--Korteweg-de Vries system with rough
data, Differ. Int. Equ. 18 (2005) 1147--1174.




\bibitem{M1999Sulem}  C. Sulem, P. L. Sulem, The nonlinear Schr\"odinger equation: self-focusing and wave
collapse. Springer--Verlag,1999.


\bibitem{M2006Tao} T. Tao,  Nonlinear dispersive equations. Local and global analysis, American Mathematical Soc., 2006.
\bibitem{MSA1993Tsutsumi}M. Tsutsumi, Well--posedness of the Cauchy problem for a coupled Schr\"odinger-KdV equation, Math. Sciences Appl. 2 (1993) 513--528.
\bibitem{E1987Tsutsumi}  Y. Tsutsumi, $L^2$-solutions for nonlinear Schr\"odinger equations and nonlinear groups, Funkcial. Ekvac. 30 (1987) 115--125.
\bibitem{DIE2010Wuyifei}Y. Wu, The Cauchy problem of the Schr\"odinger--Korteweg--deVries system, Differ. Int. Equ.
23 (2010) 569--600.


}

\end{thebibliography}
\end{document}